\documentclass[12pt,a4paper]{amsart}

\usepackage[all]{xy}
\usepackage{latexsym}
\usepackage{amssymb}
\usepackage{amsmath}
\usepackage{amsthm}
\usepackage{mathrsfs}
\usepackage{color}
\usepackage{exscale}
\usepackage[english]{babel}
\usepackage{verbatim}
\usepackage{fullpage}
\usepackage{fix-cm}
\usepackage{enumitem}
\usepackage[colorlinks, breaklinks, allcolors=blue]{hyperref} 

\makeatletter
\def\@settitle{\begin{flushleft}%
  \baselineskip14\p@\relax
    \normalfont\Large\bf%<- NEW
  \@title
  \end{flushleft}%
}
\def\section{\@startsection{section}{1}%
  \z@{.7\linespacing\@plus\linespacing}{.5\linespacing}%
  {\normalfont\bf}}

\def\@setauthors{%
  \begingroup
  \def\thanks{\protect\thanks@warning}%
  \trivlist
  \footnotesize \@topsep30\p@\relax
  \advance\@topsep by -\baselineskip
  \item\relax
  \author@andify\authors
  \def\\{\protect\linebreak}%
  \larger\sc\authors%
  \ifx\@empty\contribs
  \else
    ,\penalty-3 \space \@setcontribs
    \@closetoccontribs
  \fi
  \endtrivlist
  \endgroup
}

\makeatother

\swapnumbers
\theoremstyle{plain}
\newtheorem{theorem}{Theorem}[section]
\newtheorem{lemma}[theorem]{Lemma}
\newtheorem{proposition}[theorem]{Proposition}
\newtheorem{corollary}[theorem]{Corollary}

\theoremstyle{definition}
\newtheorem{definition}[theorem]{Definition}
\newtheorem{remark}[theorem]{Remark}
\newtheorem{example}[theorem]{Example}

\newcommand{\CC}{\mathbb{C}}
\renewcommand{\AA}{\mathbb{A}}
\newcommand{\ZZ}{\mathbb{Z}}
\newcommand{\QQ}{\mathbb{Q}}

\newcommand{\PP}{\mathbb{P}}

\newcommand{\inv}{{}^{-1}}
\newcommand{\<}{\langle}
\renewcommand{\>}{\rangle}

\newcommand{\fp}{\mathfrak{p}}

\newcommand{\fv}{\mathfrak{v}}
\newcommand{\fr}{\mathfrak{r}}

\newcommand{\fj}{\mathfrak{j}}
\newcommand{\fq}{\mathfrak{q}}

\DeclareMathOperator{\Hom}{Hom}

\DeclareMathOperator{\Span}{span}

\DeclareMathOperator{\Exp}{Exp}

\DeclareMathOperator{\Ord}{ord}

\DeclareMathOperator{\diag}{diag}
\DeclareMathOperator{\op}{op}

\newcommand{\Div}{{\mathcal D}}
\newcommand{\Fan}{{\mathscr F}}

\newcommand{\Colemb}{{\mathcal F}}
\newcommand{\Cone}{{\mathcal C}}
\renewcommand{\O}{{\mathcal O}}

\newcommand{\SL}{\mathrm{SL}}
\newcommand{\GL}{\mathrm{GL}}

\newcommand{\SO}{\mathrm{SO}}
\newcommand{\Bl}{\mathrm{Bl}}

\newcommand{\M}{{\mathcal M}}
\newcommand{\T}{{\mathcal T}}

\usepackage{parskip}

\usepackage{microtype}

\begin{document}
\title{The Knop-Luna-Vust theory of spherical embeddings, extended to non-reductive groups}
\author{Guido Pezzini}
\address{Dipartimento di Matematica ``G.\ Castelnuovo'', Universit\`a degli Studi di Roma ``La Sapienza'', Piazzale A.\ Moro 5, 00185 Roma, Italy}
\email{pezzini@mat.uniroma1.it}
\subjclass[2000]{14M27, 14M17, 14J50}

\begin{abstract}
We extend the theory of spherical embeddings to actions of connected non-reductive groups. This generalization is formally very similar to the usual reductive case:  equivariant embeddings are described essentially by collections of convex cones in a rational vector space.

We also show some new relationship between the combinatorics emerging in this case and the properties of the unipotent radical of the acting group. Finally, we apply our techniques to prove a characterization of log homogeneous varieties.
\end{abstract}

\maketitle
\section{Introduction}

In 1983 Luna and Vust developed a general theory of equivariant embeddings of varieties that are homogeneous under the action of a connected reductive algebraic group $G$, in the paper~\cite{LV83}. This theory has been reformulated and simplified by F.\ Knop in 1991 in the paper~\cite{Kn91} in the special case of spherical varieties, which are a generalization of flag varieties, symmetric varieties, and also of toric varieties.

Formally, there is an evident similarity with the classification of toric varieties by means of their fans, which are finite collections of polyhedral convex cones in a rational vector space and where each cone corresponds to an orbit. For a spherical homogeneous variety $Y$, any $G$-equivariant open embedding $X$ of $Y$ is described by a similar finite set $\Fan(X)$, called {\em colored fan}.

Its elements are still polyhedral convex cones, but they are called {\em colored cones} because they are equipped (in slightly loose terms) with marked points, choosen among a finite list of special points which depend only on $Y$. In addition, the relative interior of each colored cone must intersect a given polyhedral cone, called {\em the valuation cone}, which again depends only on $Y$. The points of this cone correspond to discrete valuations of the field $\CC(Y)$ that are $G$-invariant, for example the valuation given by a $G$-stable prime divisor in some embedding.

In this paper we extend this theory to the case where $Y$ is homogeneous under the action of a connected linear algebraic group $P$, not necessarily reductive. The basic assumption is still that $Y$ be spherical, under the action of a Levi subgroup $G$ of $P$. So the standard theory of spherical embeddings is available, and one must take into account the action of the bigger group $P$.

The end result is that the theory looks quite similar to the classical reductive case. Given a $P$-equivariant embedding $X$ of $Y$, the first task is to study its $P$-orbits. This is equivalent to studying the subfamily $\Fan(X)^\uparrow$ of the colored cones that correspond to $G$-orbits that have $P$-stable closure (see Proposition~\ref{prop:PequivariantLorbit} below). We introduce the notion of {\em colored $P$-fan} based on the properties of $\Fan(X)^\uparrow$.

Then, several statements of~\cite{Kn91} generalize smoothly with everything now reformulated in terms of $\Fan(X)^\uparrow$ instead of $\Fan(X)$. The main result of the paper, Theorem~\ref{thm:classification}, states that $P$-equivariant embeddings up to isomorphisms are classified by colored $P$-fans. The condition for the identity $Y\to Y$ to extend to a $P$-equivariant map $X_1\to X_2$ between two $P$-equivariant embeddings is the same as in the reductive case, but expressed only in terms of $\Fan(X_1)^\uparrow$ and $\Fan(X_2)^\uparrow$.

The valuation cone $V_G(Y)$ is also replaced by the subset $V_P(Y)$ corresponding to $P$-invariant valuations. This is again a convex polyhedral cone, and it is used also to describe $P$-equivariant maps between different $P$-homogeneous varieties, in the same way as in the reductive case.

Some parts of this theory had been already developed in our previous paper~\cite{Pe17}, including the construction of the cone $V_P(Y)$, so here we build on that work to complete the picture with the full embedding theory. On the other hand, it is remarkable that several proofs in this paper can simply follow the lines of corresponding proofs in Knop's paper~\cite{Kn91}, which essentially turn out to work also with non-reductive groups.

Beyond these many similarities, a novelty is discussed in Section~\ref{s:PuSigma}: we show that the shape of the $P$-invariant valuation cone $V_P(Y)$ is related to some properties of the action of the unipotent radical $P^u$ on $Y$. 

This relationship has applications even in the reductive case. Indeed, suppose that a $G$-homogeneous variety $Y$ admits a $G$-equivariant map $\varphi\colon Y\to G/Q$ to a partial flag variety $G/Q$, so $Q$ is a parabolic subgroup of $G$. Then the equality $V_G(Y)=V_Q(\varphi\inv(eQ))$ was proved in~\cite{Pe17}, and in this paper we use it to relate $V_G(Y)$ to the action of the unipotent radical $Q^u$ on $Y$.

This provides new informations on the cone $V_G(Y)$ from the knowledge of a generic stabilizer of $Y$. In general, computing $V_G(Y)$ from the stabilizer is reputed to be a rather difficult problem, solved only under additional assumptions (see Avdeev's paper~\cite{Av24}) or in special cases with laborious case-by-case considerations (see e.g. our paper~\cite{BP15} with P.\ Bravi).

A last analogy with the reductive case is discussed in Section~\ref{s:tails}. For an affine spherical $G$-variety $X$, the valuation cone has an important relationship with the multiplication in the ring of regular functions $\CC[X]$. We prove in Theorem~\ref{thm:Ptails} that if the $G$-action on $X$ extends to $P$ then a similar relationship holds with the cone of $P$-invariant valuations, but where the action of $P$ on $\CC[X]$ is considered together with the multiplication.

Using this result we are able to check in examples that $V_P(X)$ is not always cosimplicial: its maximal faces can have linearly dependent equations, see Example~\ref{ex:noncosimplicial}. This is another relevant novelty, because on the contrary $V_G(X)$ is always a Weyl chamber of a Weyl group (thanks to an important theorem of Brion's), generalizing the little Weyl group of symmetric varieties to arbitrary spherical $G$-varieties. So, under this point of view there is no analogue of the little Weyl group for the action of $P$.

In Section~\ref{s:Plog} we apply our theory to log homogeneous varieties, introduced by Brion in~\cite{Br07}. We first recall how they are related to the so-called {\em regular embeddings} (see the paper~\cite{BDP90} by Bifet, De Concini, Procesi), thanks to results of Brion's. Then we prove a characterization of varieties that are log homogeneous under the action of $P$ in Theorem~\ref{thm:Plog}.

\subsection*{Acknowledgments}
 I thank Friedrich Knop, Bart Van Steirteghem, and Simone Diverio for stimulating discussions and suggestions.

\section{Notations and basic assumptions}\label{s:notations}

%If $G$ is an affine algebraic group then we denote by $\Chars(G)$ the group of its characters, by $G^u$ its unipotent radical, by $G^\circ$ the connected component containing the neutral element $e\in G$. If $K$ is a subgroup of $G$ and $g\in G$, we set ${}^gK= gKg^{-1}$. If $V$ is a $G$-module then $V^{(G)}$ denotes the set of common eigenvectors of all elements of $G$ (in particular $0\notin V^{(G)}$), and if $\chi\in\Chars(G)$ then we set
%\[
%V^{(G)}_\chi = \{ v\in V\smallsetminus \{0\} \;|\; gv = \chi(g)v \;\forall g\in G \}
%\]
%Unless otherwise stated, we denote the Lie algebra of a group with the corresponding lower case fraktur letter.

In this paper all algebraic groups and varieties are defined over the field $\CC$ of complex numbers and varieties are always assumed to be irreducible. Subgroups of algebraic groups are always assumed to be closed.

We fix a connected linear algebraic group $P$, a Levi (i.e.\ maximal connected reductive) subgroup $G\subseteq P$, a Borel subgroup $B$ of $G$, a maximal torus $T$ of $B$, and we denote by $B_-$ the opposite Borel subgroup with respect to $T$. The set of simple roots of $G$ corresponding to these choices will be denoted by $S$.

In this paper we will consider varieties $X$ that are {\em spherical} under the action of $G$, i.e., such that $X$ is normal and has an open $B$-orbit. For details about the definitions and results mentioned in the rest of this section we refer to~\cite{Kn91}.

We will denote by $\Div(X)^B$ the finite set of $B$-stable prime divisors of $X$; the sets $\Div(X)^P$ and $\Div(X)^G$ are defined analogously. We define the set
\[
\Delta(X)=\Div(X)^B\smallsetminus \Div(X)^G
\]
and call its elements the {\em colors} of $X$.

We denote by $\CC(X)^{(B)}$ the set of all $B$-eigenvectors in $\CC(X)$, and we denote by $\Xi(X)$ the group of their $B$-eigenvalues $\chi\colon B\to \CC^*$. It is a free abelian group of finite rank, which is also called the {\em rank} of $X$. We define the vector space
\[
N(X) =  \Hom_\ZZ(\Xi(X),\QQ).
\]
The usual pairing between $v\in N(X)$ and $\xi\in \Xi(X)$ will be denoted by $\<\xi,v\>$.

The value of any discrete valuation%
\footnote{By this we mean a valuation $v\colon \CC(X)\smallsetminus\{0\}\to \QQ$ that is zero on the constants and has discrete image.} $v$ of $\CC(X)$ on a $B$-eigenvector $f\in \CC(X)^{(B)}$ depends only on the $B$-eigenvalue of $f$, thus $v$ yields a well-defined element $\rho(v)$ of $N(X)$. This applies in particular to the valuation given by any $B$-stable prime divisor $D$ of $X$, and in this case we denote simply by $\rho(D)$ the corresponding element of $N(X)$.

%If $K$ is the stabilizer in $L$ of a point in the open $L$-orbit of $X$, then $X$ is called an {\em embedding} of its open orbit $L/K$. Intersecting with $L/K$ gives a bijection between $\Delta_L(X)$ and $\Div(L/K)^B$.

There are only finitely many $G$-orbits $Z$ in $X$, each is itself spherical, and one associates with any such $Z$ the polyhedral (i.e.\ finitely generated) convex cone $\Cone_Z\subseteq N(X)$ generated by $\rho(D)$ for all $D\in\Div(X)^B$ such that $D\supseteq Z$. The {\em colored cone} associated with $Z$ is the couple $(\Cone_Z,\Colemb_Z)$ where $\Colemb_Z$ is the set of the colors of $X$ that contain $Z$. It turns out that for each $D\in\Div(X)^G$ containing $Z$, the element $\rho(D)$ generates an extremal ray of the cone $\Cone_Z$ not containing the image of any color.

The set of the colored cones of all $G$-orbits of $X$ is called the {\em colored fan} of $X$ and is denoted here by $\Fan(X)$. Two $G$-orbits $Z$, $W$ satisfy $\overline W\supseteq Z$ if and only if $\Cone_W$ is a face of $\Cone_Z$ and $\Colemb_W=\Colemb_Z\cap\rho^{-1}(\Cone_W)$. In this case we say that $(\Cone_W,\Colemb_W)$ is a {\em colored face} of $(\Cone_Z,\Colemb_Z)$.

The variety $X$ is called {\em $G$-toroidal} if any element of $\Div(X)^B$ containing a $G$-orbit is $G$-stable.

The set of $G$-invariant discrete valuations on $\CC(X)$ is identified via the map $\rho$ with its image $V_G(X)$ in $N(X)$, and $V_G(X)$ is a polyhedral convex cone of maximal dimension called the {\em valuation cone}. Here we will call it also the {\em $G$-valuation cone}; it satisfies $V_G(X)=N(X)$ if and only if a generic stabilizer of $X$ in $G$ contains a maximal unipotent subgroup of $G$.

The dual cone of $-V_G(X)$ in $\Xi(X)$ is generated by a set $\Sigma_G(X)$ of primitive, linearly independent elements. They are called the {\em spherical roots}, in this paper we add ``{\em of the $G$-action on $X$}''.

Finally we recall that there are combinatorial notions of ``abstract'' colored cones and fans in $N(X)$. They do not depend on $X$ in itself but involve the sets $V_G(X)$ and $\Delta(X)$, where the latter is considered as an abstract set equipped with the map $\rho$. They are used to classify all spherical $G$-varieties that contain the same homogeneous space $G/K$ as their open $G$-orbit.

The definition is as follows. A couple $(\Cone,\Colemb)$ is a {\em colored cone} in $N(X)$ if $\Colemb$ is a subset of $\Delta(X)$, the set $\Cone$ is a convex cone in $N(X)$ generated by $\rho(\Colemb)$ and finitely many elements of $V_G(X)$, and the relative interior $\Cone^\circ$ of $\Cone$ intersects $V_G(X)$. It is {\em strictly convex} if $\Cone$ is strictly convex and $\rho(\Colemb)$ does not contain $0$.

A {\em colored subspace} is a colored cone $(\Cone,\Colemb)$ such that $\Cone$ is a vector subspace. A {\em colored face} of a colored cone $(\Cone,\Colemb)$ is a colored cone $(\Cone',\Colemb')$ such that $\Cone'$ is a face of $\Cone$ and $\Colemb' = \Colemb\cap \rho\inv(\Cone')$.

A {\em colored fan} in $N(X)$ is a non-empty finite set $\Fan$ of colored cones in $N(X)$ such that $\Fan$ contains every colored face of every element of $\Fan$, and for every $v\in V_G(X)$ there is at most one $(\Cone,\Colemb)\in\Fan$ such that $v\in \Cone^\circ$. A colored fan is {\em strictly convex} if all its elements are strictly convex.

\section{Spherical $P$-varieties and the $P$-valuation cone}

We will work with spherical $G$-varieties such that the action of $G$ extends to the group $P$. Clearly they will have an open $P$-orbit, itself a spherical $G$-variety, which will be our starting point.

Therefore from now on through the rest of the paper we fix a closed subgroup $H\subseteq P$ such that $P/H$ is spherical.

\begin{definition}
\begin{enumerate}
\item A {\em $P$-embedding} of $P/H$ is a normal $P$-variety $X$ together with a $P$-equivariant open embedding $P/H\hookrightarrow X$. In this case we identify $P/H$ with its image in $X$, and we single out tacitly the point $x_0\in X$ corresponding to the point $eH\in P/H$.
\item  A $P$-embedding $X$ of $P/H$ is {\em $P$-toroidal} if any element of $\Div(X)^B$ containing a $P$-orbit is $P$-stable, and $X$ is called {\em $P$-simple} if it contains a unique closed $P$-orbit.
\item A {\em morphism of $P$-embeddings} of $P/H$ is a $P$-equivariant regular map $f\colon X\to X'$ extending the identity on $P/H$.
\end{enumerate}
\end{definition}

We recall the following definition.
\begin{definition}[{\cite[Definition~6.12]{Pe17}}]
\begin{enumerate}
\item Let $X$ be a complete $P$-toroidal embedding of $P/H$. We define the {\em $P$-valuation cone}
\[
V_P(P/H) = \bigcup_Z \Cone_Z
\]
where $Z$ varies in the set of $G$-orbits of $X$ such that $\overline Z$ is $P$-stable.
\item We define the set $\Sigma_P(P/H)$ of {\em spherical roots of the $P$-action on $P/H$} as the minimal set of primitive elements of $\Xi(P/H)$ such that
\[
V_P(P/H) = \{ v \in N(P/H) \;|\; \langle \sigma,v\rangle \leq 0 \;\forall\sigma\in \Sigma_P(P/H)\}.
\]
\item If $X$ is any $P$-embedding of $P/H$, we set
\[
V_P(X) = V_P(P/H),\qquad
\Sigma_P(X)=\Sigma_P(P/H)
\]
and we say that $X$ is {\em $P$-horospherical} if $V_P(X) = N(X)$, equivalently $\Sigma_P(X)=\varnothing$.
\end{enumerate}
\end{definition}

The definition of $V_P(P/H)$ does not depend on the choice of the complete $P$-toroidal embedding thanks to~\cite[Theorem~6.11]{Pe17}, and $V_P(P/H)$ is a convex polyhedral cone in $N(P/H)$ contained in $V_G(P/H)$ thanks to~\cite[Theorem~6.10]{Pe17}.

The inclusion $V_P(P/H)\subseteq V_G(P/H)$ corresponds under certain assumptions on $P$ to an inclusion $\Sigma_P(P/H)\supseteq \Sigma_G(P/H)$, see~\cite[Corollary~8.12]{Pe17}. In addition, we will see in Section~\ref{s:PuSigma} that the elements in $\Sigma_P(P/H)\smallsetminus \Sigma_G(P/H)$ have an important relationship with the unipotent radical $P^u$ of $P$.

In~\cite{Pe17} we were led to consider $V_P(X)$ because of the following proposition. It describes a situation where the $G$-valuation cone of $X$ can be recovered as the $Q$-valuation cone of a subvariety, where $Q$ is a parabolic subgroup of $G$.

\begin{proposition}[{\cite[Theorem~13.1]{Pe17}}]\label{prop:fiberoverflag}
Let $X$ be a spherical $G$-variety equipped with a $G$-equivariant regular map $\varphi\colon X\to G/Q$ where $Q$ is a parabolic subgroup of $G$. We may suppose $Q$ contains $B_-$, and let $M\subseteq Q$ be a Levi subgroup of $Q$ containing $T$. Then the fiber $X'=\varphi\inv(eQ)$ is a spherical $M$-variety satysfying the equalities
\[
V_G(X) = V_Q(X'),\qquad
\Sigma_G(X)=\Sigma_Q(X').
\]
\end{proposition}

\section{The classification of $P$-embeddings}

\begin{lemma}\label{lemma:uniqueLorbits}
Let $X$ be a $P$-embedding of $P/H$. Then the following assertions hold.
\begin{enumerate}
\item Each $P$-orbit $Z$ of $X$ contains a unique $G$-orbit $Z^\uparrow$ such that $Z^\uparrow$ is open in $Z$, and a unique $G$-orbit $Z^\downarrow$ such that $Z^\downarrow$ is closed in $Z$.
\item For all $P$-orbits $Z$ of $X$, all $G$-stable prime divisors of $Z$ contain $Z^\downarrow$.
\end{enumerate}
\end{lemma}
\begin{proof}
We prove the first assertion. Let $Z\subseteq X$ be a $P$-orbit. It is a locally closed $G$-stable irreducible subvariety of $X$, whence $Z$ is $G$-spherical. Being a $P$-orbit, it can be identified with a homogeneous space $Z=P/K$, let $Q\subseteq P$ be a parabolic subgroup minimal containing $K$. The homogeneous variety $P/Q$ is complete and $Q$ contains the unipotent radical of $P$, hence $P/Q$ is a single $G$-orbit.

Let $W\subseteq Z$ be a $G$-orbit closed in $Z$, and consider the natural projection $\varphi\colon Z\to P/Q$. Since $P/Q$ is $G$-homogeneous, we have that $W$ is mapped surjectively to $P/Q$ via $\varphi$ and $W$ intersects the fiber $Z'=\varphi\inv(eQ)$.

The subgroup $K$ is {\em regularly embedded} in $Q$, i.e., the unipotent radical $K^u$ is contained in $Q^u$, see~\cite[Definition~3.6]{Ti11}. This holds by the minimality in the definition of $Q$, because this minimality implies that the image of $K$ in $Q/Q^u$ is reductive. We conclude that $Z'\cong Q/K$ is affine, by~\cite[Theorem~3.7]{Ti11}. Moreover $Z'$ is spherical under the action of a Levi subgroup $M$ of $Q$, as we have recalled in Proposition~\ref{prop:fiberoverflag}. We may suppose $M$ is contained in $G$; being affine and spherical, the variety $Z'$ has a unique open $M$-orbit therefore a unique closed $M$-orbit by GIT.

The intersection $W\cap Z'$ is a closed and $M$-stable subset of $Z'$, hence it contains the unique closed $M$-orbit of $Z'$. This shows that all closed $G$-orbits of $Z$ intersect, therefore there is only one such orbit $Z^\downarrow$. Finally, the variety $Z$ is $G$-spherical therefore it has a unique open $G$-orbit $Z^\uparrow$.

The second assertion is a direct consequence of the first, since $Z$ has a unique closed $G$-orbit.
\end{proof}

We extend the notations introduced in the above lemma.
\begin{definition}
Let $X$ be a $P$-embedding of $P/H$ and let $Z\subseteq X$ be a non-empty subset contained in a single $P$-orbit $PZ$. We denote by $Z^\uparrow$ the unique $G$-orbit that is dense in $PZ$, and by $Z^\downarrow$ the unique $G$-orbit that is closed in $PZ$.
\end{definition}

\begin{remark}
If $Z,W$ are $P$-orbits of $X$ such that $W\subseteq \overline Z$ holds, then of course we have $W^\uparrow\subseteq \overline{Z^\uparrow}$. On the other hand, in general there is no inclusion relation between the closures of $W^\downarrow$ and $Z^\downarrow$, as the next example shows.
\end{remark}

\begin{example}\label{ex:P2}
Let $G=B=T$ be the usual maximal torus of $\GL(3)$ consisting of the diagonal matrices, and let $P$ be the subgroup of upper-triangular matrices in $\GL(3)$. Let $X=\PP^2$ with the usual linear action of $\GL(3)$. Then $X$ is spherical and the action extends to $P$. The $G$-stable prime divisors are the three coordinate hyperplanes $D_i$, where the notation is such that $D_i\not\ni [e_i]$ and $(e_1,e_2,e_3)$ is the canonical basis of $\CC^3$. The prime divisor $D_3$ is the only $P$-stable one among them. There are three $P$-orbits:
\begin{enumerate}
\item the single point $Z_0 = \{[e_1]\}$, it is a single $G$-orbit thus $Z_0=Z_0^\downarrow$;
\item the affine line $Z_1 = D_3\smallsetminus\{[e_1]\}$, it is the union of two $G$-orbits and we have $Z_1^\downarrow=\{[e_2]\}$;
\item the open $P$-orbit $P/H=Z_2 = \PP^2\smallsetminus D_3$, which has three $G$-orbits and satisfies $Z_2^\downarrow=\{[e_3]\}$.
\end{enumerate}
Notice that $Z_0$ is in the closure of $Z_1$ which is in the closure of $Z_2$ but there is no inclusion relation between the closures in $X$ of the $G$-orbits $Z_i^\downarrow$ which are all single points.

The vector space $N(X)$ has dimension $2$ and there is no color, the group $G$ is a torus so $X$ is $G$-horospherical. The colored fan of $X$ has three rays, generated by the valuations $v_i$ of the divisors $D_i$, and three $2$-dimensional cones, each spanned by two of the rays. 
\begin{figure}[htbp]
\centering
\[
\xy
(0,0);(-20,0)
**\dir{-};
?>*@{>};
(0,0);(0,-20)
**\dir{-};
?>*@{>};
(0,0);(20,20)
**\dir{-};
?>*@{>};
(3,-20)*{\scriptstyle{v_1}};
(-20,3)*{\scriptstyle{v_2}};
(22,17)*{\scriptstyle{v_3}};
(-16,-9)*{\scriptstyle{\text{cone of the $G$-orbit}}};
(-16,-13)*{\scriptstyle{\text{closed in } P/H}};
(0,0)*\cir<10pt>{ul^d};
(0,0)*\cir<8pt>{l_u};
(0,0)*\cir<12pt>{dr_l};
\endxy
\]
\caption{Colored fan of $\PP^2$ \label{fig:P2}}
\end{figure}
The colored cone of the $G$-orbit closed in $P/H$ is generated by the valuations of the $G$-stable prime divisors that intersect $P/H$, i.e., it is the cone generated by $v_1$ and $v_2$.

Let us compute the $P$-spherical roots. The embedding $X$ of the open $P$-orbit $Z_2=P/H$ is not $P$-toroidal, because the $B$-stable prime divisor $D_2$ is the closure of a $B$-stable prime divisor of $P/H$, and $D_2$ contains the $P$-orbit $Z_0$. To obtain a $P$-toroidal completion of $P/H$ it is enough to blow up the point $Z_0$, so set $\widetilde X = \Bl_{[e_1]}(X)$.

The fan acquires an additional ray is generated by $v_4 = v_1+v_3$, and the cone generated by $v_2$ and $v_3$ is replaced by the two cones generated resp.\ by $v_3,v_4$ and by $v_2,v_4$. It is easy to describe the four $G$-orbits with $P$-stable closure, and obtain that $V_P(P/H)$ is generated by $v_3,v_4$.

\begin{figure}[htbp]
\centering
\[
\xy
(0,0);(-20,0)
**\dir{-};
?>*@{>};
(0,0);(0,-20)
**\dir{-};
?>*@{>};
(0,0);(20,20)
**\dir{-};
?>*@{>};
(0,0);(20,0)
**\dir{-};
?>*@{>};
(3,-20)*{\scriptstyle{v_1}};
(-20,3)*{\scriptstyle{v_2}};
(22,17)*{\scriptstyle{v_3}};
(22,-3)*{\scriptstyle{v_4}};
(-16,-9)*{\scriptstyle{\text{cone of the $G$-orbit}}};
(-16,-13)*{\scriptstyle{\text{closed in } P/H}};
(0,0)*\cir<10pt>{ul^d};
(0,0)*\cir<8pt>{l_u};
(0,0)*\cir<12pt>{d_l};
(0,0)*\cir<14pt>{dr_d};
(16,6)*{\scriptstyle{V_P(P/H)}};
\endxy
\]
\caption{Colored fan of $\Bl_{[e_1]}\left(\PP^2\right)$ \label{fig:BlP2}}
\end{figure}

Notice the properties:
\begin{enumerate}
\item all colored cones $(\Cone, \Colemb)$ of the fans of $X$ and of $\widetilde X$ are generated by $\rho(\Colemb)(=\varnothing)$, some elements of $\rho\left(\Div(P/H)^G\right)=\{v_1,v_2\}$ and some elements in $V_P(P/H)$;
\item the colored cones of $G$-orbits with $P$-stable closures are exactly those whose relative interior intersects $V_P(P/H)$.
\end{enumerate}
\end{example}

We go back to the classification of $P$-equivariant embeddings of $P/H$. We will prove that the $P$-invariant discrete valuations of $\CC(P/H)$ correspond to the points of $V_P(P/H)$. A first approximation of this fact is easier and we prove it in the next proposition.

\begin{proposition}\label{prop:VPisPinvariant}
Let $X$ be a $P$-embedding of $P/H$ and let $Z\subseteq X$ be a $G$-orbit. If the relative interior of $\Cone_Z$ intersects $V_P(P/H)$ then $\overline Z$ is $P$-stable.
\end{proposition}
\begin{proof}
Let $X'$ be a $P$-equivariant completion of $X$ and let $X''$ be a complete $P$-toroidal embedding of $P/H$ such that the identity of $P/H$ extends to a proper $P$-equivariant regular map $\varphi\colon X''\to X'$. Such $X'$ exists thanks to a theorem of Sumihiro's~\cite{Su74}, and $X''$ exists thanks to~\cite[Lemma~5.2]{Pe17}. The union of all cones of $\Fan(X'')$ contains $V_G(P/H)$ because $X''$ is complete~\cite[Theorem~4.2]{Kn91}, so this union contains $V_P(P/H)$ too. By~\cite[Theorem~4.1]{Kn91} (see also Section~\ref{s:morphisms} below), this implies that there is a $G$-orbit $W\subseteq X''$ such that $\Cone_W$ is contained in $\Cone_Z$, and such that the relative interior of $\Cone_W$ intersects the relative interior of $\Cone_Z$ in some point of $V_P(P/H)$.

Again~\cite[Theorem~4.1]{Kn91} implies that $\varphi(W)$ is equal to $Z$. On the other hand, by the definition of $V_P(P/H)$ we deduce that $\Cone_W$ is contained in $V_P(P/H)$, because the relative interiors of the cones in the fan of $X''$ are allowed to overlap, but not on points of $V_G(P/H)$, as we have recalled in Section~\ref{s:notations}.

Thanks to~\cite[Proposition~6.3]{Pe17} the closure $\overline W$ is $P$-stable. This shows that the closure of $Z$ is $P$-stable, equivalently in $X'$ or in $X$.
\end{proof}

\begin{definition}\label{def:Pequivariantcones}
\begin{enumerate}
\item A colored cone or a colored subspace $(\Cone, \Colemb)$ in $N(P/H)$ is called {\em $P$-equivariant} if $\Cone$ is generated as a convex cone by $\rho(\Colemb)$, a subset of $\rho\left(\Div(P/H)^G\right)$, and finitely many elements of $V_P(P/H)$.
\item A colored cone $(\Cone, \Colemb)$ in $N(P/H)$ is called {\em $P$-invariant} if it is $P$-equivariant and the relative interior of $\Cone$ intersects $V_P(P/H)$.
\item A {\em colored $P$-fan} in $N(P/H)$ is a non-empty finite set $\Fan$ of $P$-invariant colored cones such that $\Fan$ contains every $P$-invariant face of every element of $\Fan$, and for every $v\in V_G(X)$ there is at most one $(\Cone,\Colemb)\in\Fan$ such that $v\in \Cone^\circ$. A colored $P$-fan is {\em strictly convex} if all its elements are strictly convex.
\item For any colored cone $(\Cone, \Colemb)$ in $N(P/H)$ we denote by $(\Cone,\Colemb)^\uparrow$ the unique maximal $P$-invariant colored face of $(\Cone,\Colemb)$.
\item We denote by $\Fan^\uparrow$ the subset of $P$-invariant colored cones of any colored fan $\Fan$ in $N(P/H)$.
\end{enumerate}
\end{definition}

The above definition is motivated by Proposition~\ref{prop:PequivariantLorbit} below.

\begin{remark}
We point out that a colored $P$-fan is not necessarily a colored fan according to the standard definition. The reason is that a $P$-invariant colored cone can have colored faces that are not $P$-invariant, i.e., such that their relative interiors intersect $V_G(P/H)$ but not $V_P(P/H)$. Thanks to Proposition~\ref{prop:PequivariantLorbit}, this corresponds to the fact that $G$-orbits with $P$-stable closure can be contained in non-$P$-stable $G$-orbit closures, see Example~\ref{ex:P2}.
\end{remark}

We begin the classification by proving the uniqueness part for $P$-simple $P$-embeddings.

\begin{proposition}\label{prop:Psimpleuniqueness}
Let $X$ be a $P$-simple $P$-embedding of $P/H$, with closed $P$-orbit $Z$. The colored cone $(\Cone_{Z^\uparrow},\Colemb_{Z^\uparrow})$ determines $X$ uniquely up to isomorphisms of $P$-embeddings.
\end{proposition}
\begin{proof}
The proof is essentially the same of~\cite[Theorem~2.3]{Kn91}. Let $X'$ be a $P$-simple $P$-embedding of $P/H$, with closed $P$-orbit $Z'$, and assume $(\Cone_{Z^\uparrow},\Colemb_{Z^\uparrow})=(\Cone_{(Z')^\uparrow},\Colemb_{(Z')^\uparrow})$.

Define
\[
X_0 = X\smallsetminus\bigcup_{\substack{D\in\Div(X)^B,\\D\not\supseteq Z}} D
\]
and $X'_0$ similarly.

Recall that the colored cone $(\Cone_{Z^\uparrow},\Colemb_{Z^\uparrow})$ determines, up to strictly positive rational multiples, the valuations of the $B$-stable prime divisors of $X$ containing $Z^\uparrow$ (equivalently, containing $Z$). Indeed, the valuations of the $G$-stable prime divisors containing $Z$ generate precisely those extremal rays of $\Cone_{Z^\uparrow}$ not containing elements of $\rho(\Colemb_{Z^\uparrow})$.

This implies that $X_0$ and $X_0'$ have the same $B$-semiinvariant rational functions. The open subset
\[
X_1 = (P/H) \smallsetminus\bigcup_{\substack{D\in\Div(X)^B,\\D\not\supseteq Z}} D
\]
of $P/H$ is $B$-stable, and it is affine thanks to~\cite[Theorem~2.1]{Kn91}. Then we have
\[
\CC[X_0] = \{f\in\CC[X_1]\;|\; v(f)\geq 0\;\forall v\in \Cone_Z\}=\CC\left[X_0'\right].
\]
We conclude that the identity $P/H\to P/H$ extends to a regular map $X_0\to X_0'$. Using the action of $P$, the same map extends to an isomorphism of $P$-embeddings
\[
X=PX_0\to PX_0'=X'.
\]
\end{proof}

\begin{theorem}\label{thm:existencePsimple}
Let $(\Cone,\Colemb)$ be a strictly convex $P$-invariant colored cone in $N(P/H)$. Then there exists a unique (up to isomorphisms) $P$-simple $P$-equivariant embedding $X$ of $P/H$, with closed $P$-orbit $Z$, such that $(\Cone,\Colemb)=(\Cone_{Z^\uparrow},\Colemb_{Z^\uparrow})$.
\end{theorem}
\begin{proof}
The proof is similar to the one of~\cite[Theorem~3.1]{Kn91}. The uniqueness stems from Proposition~\ref{prop:Psimpleuniqueness}. Let $(\Cone,\Colemb)$ be a strictly convex $P$-invariant colored cone of $N(P/H)$, we construct the embedding $X$.

Let $\Cone^\vee$ be the dual cone of $\Cone$ in $\Xi(P/H)$ (i.e.\ the set of all elements that are non-negative on $\Cone$). By Gordan's Lemma, the set $\Cone^\vee$ is a finitely generated monoid. Let $\chi_1,\ldots,\chi_n$ be generators of $\Cone^\vee$ and for all $i\in\{1,\ldots,n\}$ choose a function $g_i\in\CC(P/H)^{(B)}$ with $B$-eigenvalue $\chi_i$. Let $D_0$ be the union of all elements of $\Div(P/H)^B$ not contained in $\Colemb$ and not contained in $\Div(P/H)^G\cap\rho^{-1}(\Cone)$.

For all $i$ the poles of $g_i$ are contained in $D_0$. The pull-back of $g_i$ in $\CC(P)$, also denoted by $g_i$, is $H$-invariant for the right translation action of $P$ on itself, and $B$-semiinvariant for the left translation action. Choose a non-zero function $f_0\in\CC[P]$ that is $B$-semiinvariant for left translation and $H$-semiinvariant for right translation, and that vanishes exactly on the inverse image of $D_0$ in $P$, with enough multiplicity so that $f_i=f_0g_i$ is a regular function on $P$. For its existence, see e.g.\ the proof of~\cite[Lemma~5.2]{Pe17}.

The functions $f_0,f_1,\ldots,f_n$ are $H$-semiinvariants for the right translation action, all with the same $H$-eigenvalue. Therefore the functions $f_0,f_1,\ldots,f_n$ generate in $\CC[P]$ a $P$-submodule $V$ (for the left translation action) which is made of analogous $H$-semiinvariants. Hence it defines naturally a $P$-equivariant map $P\to \PP(V^*)$ which descends to a $P$-equivariant map $\varphi\colon P/H\to\PP(V^*)$.

Let $\overline X$ be the closure of the image of $\varphi$, set
\[
X_0 = \left\{ x\in \overline X\;\middle|\; f_0(x)\neq 0\right\},\qquad X = P\cdot X_0
\]
and
\[
\M = \bigcup_{\chi\in \Cone^\vee} \CC(P/H)_\chi^{(B)}
\]
where $\CC(P/H)_\chi^{(B)}$ denotes the set of $B$-eigenvectors with $B$-eigenvalue $\chi$.

We claim the equality $\M = \CC[X_0]^{(B)}$ holds. The inclusion $\M \subseteq \CC[X_0]^{(B)}$ stems from the fact that the functions $g_1,\ldots,g_n$ correspond to regular functions on the principal open subset $\PP(V^*)_{f_0}$, so they restrict to regular functions on $X_0$. 

To show the other inclusion, we show that $\langle \chi,c \rangle \geq 0$ for any $c$ belonging to a given set of generators of $\Cone$ as in Definition~\ref{def:Pequivariantcones} and any $B$-character $\chi\neq0$ such that $\CC[X_0]^{(B)}_\chi$ is non-empty. Denote by $f$ an element of $\CC[X_0]^{(B)}_\chi$.

If $c=\rho(D)$ where $D\in \Colemb$ or $D\in\Div(P/H)^{G}\cap \rho\inv(\Cone)$ then $\langle \chi,c \rangle$, which is equal to $\Ord_D(f)$, is non-negative because the poles of $f$ as a rational function on $P/H$ are contained in $D_0$.

Otherwise $c$ is an element of $V_P(P/H)$. Recall that $V_P(P/H)$ is the convex cone generated by the $P$-invariant valuations of the $P$-stable prime divisors of a $P$-toroidal complete embedding of $P/H$. Therefore $c$ can be written as a linear combination of such valuations with non-negative rational coefficients. Replacing $c$ with a positive multiple we may write
\[
c = \rho(v_1)+\ldots+\rho(v_r)
\]
where $v_k$ is $P$-invariant for all $k$. Thus $\<\chi,c\> = v_1(f)+\ldots+v_r(f)$. Lift each $v_k$ to a $P$-invariant valuation of $\CC(P)$ thanks to~\cite[Lemma~7.4]{Pe17}.

For all $i\in\{0,\ldots,n\}$ choose $p_{i,1},\ldots,p_{i,m_i}\in P$ in such a way that the elements
\[
p_{0,1}f_0,p_{0,2}f_0,\ldots,p_{n,m_n-1}f_n,p_{n,m_n}f_n
\]
form a basis of $V$. On the other hand $X_0$ is again a closed subvariety of $\PP(V^*)_{f_0}$, so $f$ is the restriction to $X_0$ of a regular function on $\PP(V^*)_{f_0}$, and any such function is a polynomial in $\frac{p_{0,1}f_0}{f_0},\ldots,\frac{p_{n,m_n}f_n}{f_0}$. In other words $f$ is a linear combination of products of the form
\begin{equation}\label{eq:fmonomial}
\left(\frac{p_{0,1}f_0}{f_0}\right)^{a_{0,1}}\cdots\left(\frac{p_{n,m_n}f_n}{f_0}\right)^{a_{n,m_n}}
\end{equation}
for non-negative integers $a_{0,1},\ldots,a_{n,m_n}$.

We want to prove the inequality
\[
v_1(f)+\ldots+v_r(f)\geq 0.
\]
Since each $v_k$ is a valuation, it is enough to prove this inequality assuming that $f$ is given by a single product as in (\ref{eq:fmonomial}). Moreover, for all $k$ we have $v_k(p_{0,j}f_0/f_0)=0$ by $P$-invariance of $v_k$, so we can also assume
\[
f=\left(\frac{p_{1,1}f_1}{f_0}\right)^{a_{1,1}}\cdots\left(\frac{p_{n,m_n}f_n}{f_0}\right)^{a_{n,m_n}}.
\]
Denote $s=a_{1,1}+\ldots+a_{n,m_n}$. Up to reindexing $\chi_1,\ldots,\chi_n$ and introducing repetitions if necessary, we may assume $s\leq n$ and that $f f_0^s$ is in $V_1\cdots V_s$ where $V_j$ is the $P$-submodule of $\CC[P]$ generated by $f_j$ for all $j\in\{1,\ldots,s\}$. Lemma~7.5 of~\cite{Pe17} applied to $v_k$ and $h = ff_0^s$ yields
\[
\<\chi,c\> =  \sum_{k=1}^rv_k(f) \geq \sum_{k=1}^r(v_k(g_1)+\ldots+v_k(g_s)) = \sum_{j=1}^s(v_1(g_j)+\ldots+v_r(g_j))=
\]
\[
=\langle \chi_1,c\rangle+\ldots\langle \chi_s,c\rangle \geq 0,
\]
concluding the proof of the equality $\M = \CC[X_0]^{(B)}$. This equality also implies that $X_0$ is normal, and by construction $X$ is normal too.

We claim now that $\varphi\colon P/H\to X$ is an open embedding. Being $P$-equivariant and being $P/H$ a single orbit, we must only show that $\varphi$ is generically injective. Since $\Cone$ is strictly convex, the group generated by $\Cone^\vee$ is the whole lattice $\Xi(P/H)$, so $\Xi(P/H)=\Xi(X)$. By~\cite[Lemma~2.4]{Ga11} the map $\varphi$ restricts to an isomorphism between the open $G$-orbit of $P/H$ and the one of $X$, whence $\varphi$ is generically injective.

It remains to prove that $X$ has a unique closed $P$-orbit $Z$, and that the colored cone of $Z^\uparrow$ is $(\Cone,\Colemb)$.

Since the colored cone $(\Cone,\Colemb)$ is $P$-invariant, there exists an element $V_P(P/H)$ in the relative interior of $\Cone$. Since $V_G(P/H)$ contains $V_P(P/H)$, we can write such element as $\rho(v_0)$ where $v_0$ is a $G$-invariant valuation on $\CC(P/H)$. Then $\rho(v_0)$ is non-negative on $\CC[X_0]^{(B)}$, and by~\cite[Corollary~1.7]{Kn91} the valuation $v_0$ is also non-negative on $\CC[X_0]$. Hence $v_0$ has a center on $X_0$; closing this center in $X$ we obtain a closed subvariety $W$ that is $G$-stable because $v_0$ is $G$-invariant.

Let $U\subseteq X$ be a closed $P$-stable subvariety, we claim that $U$ contains $W$. Choose a $G$-invariant valuation $v_1$ whose center is $U$. By construction, all $P$-orbits of $X$ intersect $X_0$, therefore on the affine variety $X_0$ the valuation $v_1$ has center $U\cap X_0$. It follows that $v_1$ is non-negative on $\CC[X_0]$, thus $\rho(v_1)$ is in $\Cone$.

Suppose now $W\not\subseteq U$, by~\cite[Corollary~1.7]{Kn91} there is a function $f\in\CC[X_0]^{(B)}$ such that $v_0(f)=0$ and $v_1(f)>0$. Since $v_0$ is in the relative interior of $\Cone$ and the $B$-eigenvalue of $f$ is in $\Cone^\vee$, the condition $v_0(f)=0$ implies $v_1(f)=0$: contradiction. Therefore we have $W\subseteq U$.

This implies that there is a unique closed $P$-orbit $Z$, and that it contains $W$. The same argument shows that $W$ is contained in all $B$-stable prime divisors of $X_0$, and of course $W$ is not contained in the zero set of $f_0$. This implies $\Cone=\Cone_{W}$ and $\Colemb=\Colemb_{W}$.

By Proposition~\ref{prop:VPisPinvariant} the closure $\overline W$ is $P$-invariant, whence $W=Z^\uparrow$.
\end{proof}

\begin{corollary}\label{cor:Pvalcone}
The set $V_P(P/H)$ is the image under $\rho$ of the set of $P$-invariant discrete valuations of $\CC(P/H)$.
\end{corollary}
\begin{proof}
For all $c\in V_P(P/H)$ we can construct a $P$-invariant colored cone $(\Cone,\Colemb)$ setting $\Cone = \QQ_{\geq 0}c$ and $\Colemb=\varnothing$. By Theorem~\ref{thm:existencePsimple}, there exists a $P$-simple $P$-embedding $X$ of $P/H$ with closed $P$-orbit $Z$ such that $(\Cone,\Colemb)$ is the colored cone of $Z^\uparrow$.

We can assume $c\neq 0$, then by~\cite[Lemma~6.4]{Kn91} we have that $Z^\uparrow$ has codimension $1$ in $X$, i.e., its closure $Z$ in $X$ is a $P$-stable prime divisor. Its valuation $v_Z$ is $P$-invariant and $c$ is a positive rational multiple of $\rho(v_Z)$, this shows that $V_P(P/H)$ is contained in the image of the set of $P$-invariant valuations.

Viceversa, let $v$ be a $P$-invariant valuation, consider the couple $(\Cone,\Colemb)$ with $\Cone = \QQ_{\geq 0}\rho(v)$ and $\Colemb=\varnothing$. We cannot apply to $(\Cone,\Colemb)$ the statement of Theorem~\ref{thm:existencePsimple}, but we can apply its proof: the generator $c$ in that proof is just the point $\rho(v)$ here, so we are in the case $r=1$ and $v=v_1$. Therefore there exists a $P$-simple $P$-embedding $X$ of $P/H$ such that $(\Cone,\Colemb)$ is the colored cone of $Z^\uparrow$ where $Z$ is the unique closed $P$-orbit of $X$.

Again $Z$ is a $P$-stable prime divisor of $X$. Let $X'$ and $X''$ be as in the proof of Proposition~\ref{prop:VPisPinvariant}. The inverse image in $X''$ of the closure $\overline Z$ of $Z$ in $X'$ has at least one irreducible component $Z''$ that is a $P$-stable prime divisor of $X''$ such that $\rho(Z'')$ is a non-zero positive scalar multiple of $\rho(Z)$. This shows that $\rho(Z'')$, and thus $\rho(Z)$, is in $V_P(P/H)$.
\end{proof}

\begin{remark}
In~\cite[Definition~7.2]{Pe17} we defined $P$-equivariant colored subspaces in a slightly different way. The two definitions agree thanks to Corollary~\ref{cor:Pvalcone}.
\end{remark}

\begin{proposition}\label{prop:PequivariantLorbit}
Let $X$ be a $P$-equivariant embedding of $P/H$, and let $Z\subseteq X$ be a $G$-orbit. Then the following statements hold.
\begin{enumerate}
\item\label{prop:PequivariantLorbit:Pequivariant} The colored cone $(\Cone_Z,\Colemb_Z)$ is $P$-equivariant.
\item\label{prop:PequivariantLorbit:Pinvariant} The closure $\overline Z$ in $X$ is $P$-stable if and only if we have $Z=Z^\uparrow$ if and only if the colored cone $(\Cone_Z,\Colemb_Z)$ is $P$-invariant. In particular, the set $\Fan(X)^\uparrow$ is the subset of $\Fan(X)$ consisting of the colored cones of $G$-orbits whose closure is $P$-stable, and this induces a bijection between $\Fan(X)^\uparrow$ and the set of $P$-orbits of $X$.
\item\label{prop:PequivariantLorbit:uparrow} We have $(\Cone_{Z^\uparrow},\Colemb_{Z^\uparrow})=(\Cone_Z,\Colemb_Z)^\uparrow$. 
\item\label{prop:PequivariantLorbit:open} The colored cone $\left(\Cone_{(P/H)^\uparrow},\Colemb_{(P/H)^\uparrow}\right)$ is equal to $(\{0\},\varnothing)$ and it is the unique maximal $P$-invariant colored face of $\left(\Cone_{(P/H)^\downarrow},\Colemb_{(P/H)^\downarrow}\right)$. Hence
\[
\Cone_{(P/H)^\downarrow}\cap V_P(P/H)=\{0\}.
\]
\end{enumerate}
\end{proposition}
\begin{proof}
The cone $\Cone_Z$ is generated by $\rho(\Colemb_Z)$ together with some elements of the form $\rho(D)$ where $D$ is a $G$-stable prime divisor of $X$. In the case where $D$ is $G$-stable, then either $D$ intersect $P/H$, in which case we have $\rho(D)\in\rho\left(\Div(P/H)^G\right)$, or $D$ is $P$-stable. If $D$ is $P$-stable then $v_D$ is $P$-invariant, hence $\rho(D)$ is in $V_P(P/H)$ thanks to Corollary~\ref{cor:Pvalcone}. This shows that $(\Cone_Z,\Colemb_Z)$ is $P$-equivariant, i.e., part~(\ref{prop:PequivariantLorbit:Pequivariant}).

If the colored cone $(\Cone_Z,\Colemb_Z)$ is $P$-invariant then $\overline Z$ is $P$-stable by Proposition~\ref{prop:VPisPinvariant}. For the reverse implication we suppose $\overline Z$ is $P$-stable and we proceed as in the proof of Corollary~\ref{cor:Pvalcone}. 

More precisely, consider $X'$, $X''$, and $\varphi$ as in the proof of Proposition~\ref{prop:VPisPinvariant}. The closure $\overline Z$ in $X'$ is $P$-stable. Let $W\subseteq X''$ be a $G$-orbit that is dense in an irreducible component of $\varphi\inv(Z)$. We have $\varphi(W)=Z$ by $G$-equivariance of $\varphi$, and $W$ is maximal having this property. From these considerations and from embedding theory we deduce that $(\Cone_W,\Colemb_W)$ is a minimal element of $\Fan(X'')$ such that $\Cone_W$ is contained in $\Cone_Z$ and intersects the relative interior of $\Cone_Z$. So the relative interior of $\Cone_W$ is contained in the relative interior of $\Cone_Z$.

On the other hand the closure of $W$ is an irreducible component of $\varphi\inv\left(\overline Z\right)$, and since $P$ is connected we obtain that $\overline W$ is $P$-stable. Therefore $\Cone_W$ is contained in $V_P(P/H)$, and we conclude that the relative interior of $\Cone_Z$ intersects $V_P(P/H)$. In other words $(\Cone_Z,\Colemb_Z)$ is $P$-invariant. Part~(\ref{prop:PequivariantLorbit:Pinvariant}) is proven.

We prove part~(\ref{prop:PequivariantLorbit:uparrow}). The colored cone $(\Cone_{Z^\uparrow},\Colemb_{Z^\uparrow})$ is a colored face of $(\Cone_Z,\Colemb_Z)$ by embedding theory. Thanks to the first part of the proof, the $P$-invariant colored cones in $\Fan(X)$ correspond to the $P$-orbits of $X$, and it is enough to notice that $PZ$ is the smallest $P$-orbit containing $Z$ in its closure.

Part~(\ref{prop:PequivariantLorbit:open}) stems from~(\ref{prop:PequivariantLorbit:uparrow}).
\end{proof}

As usual, $G$-orbits in a $G$-variety have a natural partial order given by the inclusion of closures, and colored cones have a natural partial order given by being a colored face.

\begin{corollary}
Let $X$ be a $P$-embedding of $P/H$ and $Z\subseteq X$ be a $G$-orbit with $P$-stable closure. Then $W\mapsto (\Cone_W,\Colemb_W)$ is an order-reversing  bijection between the set of $G$-orbits of with $P$-stable closure of $X$ and the set of $P$-invariant colored faces of $(\Cone_Z,\Colemb_Z)$.
\end{corollary}
\begin{proof}
The corollary stems from~\cite[Lemma~3.2]{Kn91} and part~(\ref{prop:PequivariantLorbit:Pinvariant}) of the above proposition.
\end{proof}

\begin{corollary}
If $\Sigma_P(P/H)=\Sigma_G(P/H)$ then each $P$-orbit is a single $G$-orbit in any $P$-embedding of $P/H$.
\end{corollary}
\begin{proof}
If the valuations cones coincide then all colored cones of any $P$-embedding are $P$-invariant.
\end{proof}

\begin{theorem}\label{thm:classification}
The map $X\mapsto \Fan(X)^\uparrow$ induces a bijection between the set of $P$-embeddings of $P/H$, up to isomorphisms, and the set of strictly convex colored $P$-fans in $N(P/H)$.
\end{theorem}
\begin{proof}
The proof is the same as the one of~\cite[Theorem~3.3]{Kn91}. If $X$ is an arbitrary $P$-embedding of $P/H$, then for each $P$-orbit $Z$ in $X$ we recall the subset
\[
X_{Z,P} = \left\{ x\in X\;\middle|\; \overline{Px}\supseteq Z\right\}.
\]
It is open, $P$-stable, itself a $P$-simple embedding of $P/H$ with closed $P$-orbit, and such subsets form an open covering of $X$. Given two $P$-orbits $Z,Z'$ the intersection $X_{Z,P}\cap X_{Z',P}$ is the union of the open sets $X_{W,P}$ for all $P$-orbit $W$ containing the union $Z\cup Z'$. At this point, the injectivity in the statement of the theorem follows easily from Proposition~\ref{prop:Psimpleuniqueness}.

The proof of the surjectivity is exactly as in the proof of~\cite[Theorem~3.3]{Kn91}.
\end{proof}

We end this section with a criterion for an embedding to be $P$-toroidal, which is a weaker property than being $G$-toroidal.

\begin{proposition}
Let $X$ be a $P$-embedding of $P/H$. Then $X$ is $P$-toroidal if and only if all colored cones $(\Cone,\Colemb)$ in $\Fan(X)^\uparrow$ satisfy $\Cone\subseteq V_P(P/H)$ and $\Colemb=\varnothing$.
\end{proposition}
\begin{proof}
Suppose $X$ is $P$-toroidal and consider $(\Cone,\Colemb)\in\Fan(X)^\uparrow$. By Proposition~\ref{prop:PequivariantLorbit} this colored cone corresponds to a $G$-orbit $Z$ with $P$-stable closure. Any element of $\Colemb$ is a color of $X$ containing $Z$, hence containing $\overline Z$ which is a union of $P$-orbits. Since $X$ is $P$-toroidal, no such element exists so $\Colemb=\varnothing$.

This implies that $\Cone$ is generated by the valuations of the $G$-stable prime divisors containing $Z$. Again, any such prime divisor contains $\overline Z$ hence it contains a $P$-orbit, whence it must be $P$-stable. Its valuation is in $V_P(P/H)$ by Corollary~\ref{cor:Pvalcone}. We conclude that $\Cone$ is contained in $V_P(P/H)$.

We prove the converse. Let $D$ be a $B$-stable prime divisor containing a $P$-orbit $Z$. Since $D$ contains the $G$-orbit $W=Z^\uparrow$ dense in $Z$, the valuation $\rho(D)$ is involved in the colored cone $(\Cone_W,\Colemb_W)$ in the sense that either $D$ is a color in $\Colemb_W$, or $D$ is $G$-stable and $\rho(D)$ generates an extremal ray of $\Cone_W$.

But $\Colemb_W$ is empty by assumption, so $D$ is $G$-stable and $\rho(D)$ is contained in $\Cone_W$ which is contained in $V_P(P/H)$. Let $D_0$ be the $G$-orbit dense in $D$ and apply Proposition~\ref{prop:PequivariantLorbit} to the colored cone $(\Cone_{D_0},\Colemb_{D_0})$, whose relative interior intersects $V_P(P/H)$. We obtain that the closure $D$ of $D_0$ is $P$-stable.
\end{proof}

\begin{remark}
Since $\Fan(X)^\uparrow$ determines completely any given $P$-embedding $X$, it also determines uniquely the entire colored fan $\Fan(X)$. It would be natural to look for a purely combinatorial description of $\Fan(X)$ in terms of $\Fan(X)^\uparrow$. A related question is to provide a combinatorial characterization of those colored fans $\Fan$ corresponding to $G$-equivariant embeddings (of the open $G$-orbit of $P/H$) that contain $P/H$ and are actually $P$-equivariant embeddings of $P/H$.

We believe that such descriptions are possible, but they seem to require some convoluted combinatorial considerations involving the classification of smooth affine spherical variaties. For this reason these problems go beyond the scope of the present work.
\end{remark}

\section{The spherical roots and the action of $P^u$}\label{s:PuSigma}

So far, the embedding theory for spherical $P$-varieties looks quite similar to the one for $G$-varieties, with $V_G(X)$ replaced by $V_P(X)$ and $\Fan(X)$ replaced by $\Fan(X)^\uparrow$.

In this section we show a relevant difference: the spherical roots are closely related to the action of $P^u$, something that is not present in the literature on the purely reductive case. We will also show how to apply our results on $P^u$ in the classical reductive case by means of Proposition~\ref{prop:fiberoverflag}.

Recall that the Lie algebra%
\footnote{We denote Lie algebras of groups by the corresponding lower-case fraktur letter.}
$\fp^u$ of $P^u$ is a completely reducible $G$-module under the adjoint representation, and that the exponential map $\Exp\colon \fp^u\to P^u$ is a $G$-equivariant isomorphism of varieties, where $G$ acts on $P^u$ by conjugation.

\begin{definition}
\begin{enumerate}
\item Let $S^p(P/H)$ be the set of simple roots corresponding to the stabilizer%
\footnote{The correspondence is meant in such a way that $S^p(P/H)$ is the set of simple roots of a Levi subgroup of this stabilizer, which is a parabolic subgroup of $G$ containing $B$.}
in $G$ of the open $B$-orbit of $P/H$. 
\item Let $\sigma\in \Sigma_P(P/H)\smallsetminus \Sigma_G(P/H)$. We set
\[
S(\sigma) = S^p(P/H)\cup \left\{\alpha\in S\;\middle|\; \<\sigma,\alpha^\vee\> >0 \right\}
\]
and we denote by $w_{0,\sigma}$ the longest element in the corresponding parabolic subgroup of the Weyl group of $G$.
\end{enumerate}
\end{definition}

We start with a lemma on spherical roots which generalizes a known fact: for any spherical $G$-variety $X$, on each element of $\Sigma_G(X)$ at least one and at most two simple coroots of $G$ take a strictly positive value. Our version for elements of $\Sigma_P(P/H)\smallsetminus \Sigma_G(P/H)$ limits this number to at most one simple coroot. This is the expected number, for technical reasons of combinatorial nature related to Proposition~\ref{prop:fiberoverflag}.

\begin{lemma}\label{lemma:rank1}
Let $\sigma\in \Sigma_P(P/H)\smallsetminus \Sigma_G(P/H)$. Then the following statements hold.
\begin{enumerate}
\item\label{lemma:rank1:alpha} There exists at most one simple root $\alpha\in S$ satisfying $\<\sigma,\alpha^\vee\> >0$.
\item Let $X$ be a complete $P$-toroidal $P$-embedding of $P/H$, and choose a $G$-orbit $Z$ such that $\Cone_Z$ is contained in $V_P(P/H)$, has codimension $1$ in $N(P/H)$ and is orthogonal to $\sigma$. Let $J=P_z$ the stabilizer of a point $z\in Z$. Then the following statements hold.
\begin{enumerate}
\item\label{lemma:rank1:Y} The closure $\overline Z$ is $P$-stable, it has rank $1$, it is not $P$-horospherical but it is $G$-horospherical, and $P^u$ does not act trivially on $\overline Z$.
\item\label{lemma:rank1:Q} The product $Q=JP^u$ is the unique parabolic subgroup of $P$ minimal containing $J$.
\item\label{lemma:rank1:Z} There exists a unique $B$-stable prime divisor $D\subseteq P/J$ mapped dominantly to $P/Q$, and $\rho(D)$ is strictly positive on $\sigma$.
\item\label{lemma:rank1:M} We may choose $z$ so that $Q$ contains $B_-$, in which case let $M$ be the Levi subgroup of $Q$ containing $T$. Then the set of simple roots of $M$ is $S(\sigma)$.
\end{enumerate}
\end{enumerate}
\end{lemma}
\begin{proof}
Since $\Cone_Z$ is contained in $V_P(X)$, the closure $\overline Z$ is $P$-stable by~\cite[Proposition~6.3]{Pe17}. By~\cite[Proposition~6.6]{Pe17} we have $\Xi\left(\overline Z\right)=\ZZ\sigma$, so $\overline Z$ has rank $1$, and $V_P\left(\overline Z\right)$ is obtained intersecting $V_P(P/H)$ with the line generated by $\sigma$. Since $\sigma$ is in $\Sigma_P(P/H)$ the cone $V_P(P/H)$ is the half-line of $N\left(\overline Z\right)$ defined by $\langle \sigma,-\rangle\leq 0$. It follows that $\overline Z$ is not $P$-horospherical.

Consider the valuation cones $V_P\left(\overline Z\right)\subseteq V_G\left(\overline Z\right)$. Since $N\left(\overline Z\right)$ has dimension $1$, there are only two possibilities for $V_G\left(\overline Z\right)$, namely $V_P\left(\overline Z\right)$ and $N\left(\overline Z\right)$. In the first case $\Cone_Z$ is contained not only in a face of codimension $1$ of $V_P(P/H)$ but also of $V_G(P/H)$, which yields $\sigma\in\Sigma_G(P/H)$: contradiction. It follows $V_G\left(\overline Z\right)=N\left(\overline Z\right)$ and $\overline Z$ is $G$-horospherical.

It also follows that the stabilizer $G_z$ (and so also $J$) contains a maximal unipotent subgroup of $G$. The stabilizer $J$ cannot contain $P^u$, because $J$ cannot contain a maximal unipotent subgroup of $P$ by~\cite[Proposition~6.14]{Pe17}, since $\overline Z$ is not $P$-horospherical. Being a normal subgroup of $P$ and not contained in $J$, the unipotent radical $P^u$ does not act trivially on $\overline Z$. This concludes the proof of~(\ref{lemma:rank1:Y}).

Consider the natural morphism $\pi\colon P/J\to P/Q$ and the corresponding colored subspace $(\Cone_\pi,\Colemb_\pi)$. Thanks to~\cite[Theorem~7.6]{Pe17}, this subspace is $P$-equivariant. 

Since $P/J$ has rank $1$, its image $P/Q$ has rank $1$ or $0$. If $P/Q$ has rank $1$, then the map $\pi$ has finite fibers. But the fiber $\pi^{-1}(eQ)$ is the quotient $Q/J=P^u/(J\cap P^u)$, which is connected and not equal to a single point: contradiction. Therefore $P/Q$ has rank $0$, i.e.\ it is a complete $P$-homogeneous space (and also a complete $G$-homogeneous space). This means that $Q=JP^u$ is a parabolic subgroup of $P$ containing $J$, evidently the unique one that is minimal with these properties.  This concludes the proof of~(\ref{lemma:rank1:Q}).

We compute the $B$-stable and $G$-stable prime divisors of $P/J$, which is a $G$-simple (thanks to Lemma~\ref{lemma:uniqueLorbits}) embedding of its open $G$-orbit $Z$. For this, since $Z$ is $G$-horospherical we may assume that $z\in Z$ is chosen so that $G_z$ contains $B_-^u$. This implies the inclusions
\[
Q\supseteq J \supseteq M\cdot (G\cap Q^u)\supseteq B_-.
\]

We apply~\cite[Lemma~4.5]{Pe17} to $J$, $Q$, and $\pi$, obtaining that the fiber $U=\pi^{-1}(eQ)$ is a horospherical $M$-variety of rank $1$, with $\Xi(Z)=\Xi(U)=\ZZ\sigma$. To define the latter lattice we are taking the Borel subgroup $B\cap M$ of $M$.

Let $D\subseteq P/J$ be a $G$-stable prime divisor. Then its image $\pi(D)$ is the $G$-homogeneous space $P/Q$, so $D$ intersects each fiber of $\pi$ in a divisor of the fiber. In particular, the intersection $D\cap U$ is an $M$-stable divisor. A similar reasoning applies if $D$ is any $B$-stable prime divisor of $P/J$ mapped dominantly to $P/Q$, since the point $eQ\in P/Q$ is in the open $B$-orbit of $P/Q$.

The exponential map $\fp^u\to P^u$ restricts to an $M$-equivariant isomorphism $\fj\cap \fp^u\to J\cap P^u$. Then $U\cong P^u/(J\cap P^u)\cong \fp^u/(\fj\cap \fp^u)$ is a spherical module for $M$. Moreover $U$ has rank $1$, which implies that it is an irreducible (horospherical) $M$-module.

Let $f$ be a highest weight vector in the dual module $U^*$, then the zero set of $f$ on $U$ is the unique $B\cap M$-stable prime divisor of $U$. It is either $M$-stable or a color of $U$. Correspondingly, either $P/J$ has a unique $G$-stable prime divisor $D$ and no color mapped dominantly to $P/Q$, or $P/J$ has no $G$-stable prime divisors and exactly one color $D$ mapped dominantly to $P/Q$.

Since $P/J$ has no spherical roots for the action of $G$, there is a bijection
\[
\begin{matrix}
S\smallsetminus S^p(P/J) & \to & \Delta(P/J)\\
\beta & \mapsto & D_\beta
\end{matrix}
\]
such that $\rho(D_\beta)=\beta^\vee|_{\Xi(P/J)}$ for all $\beta$, and such that the stabilizer of $D_\beta$ in $G$ is the minimal parabolic subgroup of $G$ strictly containing $B$ and corresponding to the simple root $\beta$. For these facts see~\cite[Section~1.4]{Lu01}. If $D$ is a color, let $\alpha$ be the simple root such that $D_{\alpha}=D$ and thus $\rho(D)=\alpha^\vee|_{\Xi(P/J)}$.

On the other hand, all colors of $P/J$, except for $D$ if the latter is a color, are not in $\Colemb_\pi$, and they are mapped by $\pi$ to the Schubert prime divisors of $P/Q$, which is a partial flag variety for the reductive group $G$ (see~\cite[Proposition~3.3.3]{Lu01}).

This implies that the set of simple roots of $M$ is
\[
S^p(P/J) \cup\{\alpha\}
\]
if $D$ is a color, or
\[
S^p(P/J)
\]
if $D$ is $G$-stable.

The image $P/Q$ has rank $0$, hence $\Cone_\pi=N(P/J)$. The definition of $P$-equivariant colored subspaces implies that $N(P/J)$ is generated as a convex cone by a subset of $V_P(P/J)$ and by $\rho(D)$.

This yields
\[
\<\sigma,\rho(D)\> >0
\]
completing the proof of~(\ref{lemma:rank1:Z}).  If $D$ is a color this yields also
\[
\<\sigma,\alpha^\vee\> >0.
\]

Let $D_\beta$ be a color of $P/J$, and if $D$ is a color we assume in addition $D_\beta\neq D$, i.e.\ $\beta\neq\alpha$. We claim that we have $\<\sigma,\rho(D_\beta)\>\leq 0$, equivalently $\<\sigma,\beta^\vee\> \leq 0$.

For sake of contradiction, suppose $\<\sigma,\rho(D_\beta)\> > 0$ and set $\Cone=N(P/J)$, $\Colemb=\{D_\beta\}$. The couple $(\Cone,\Colemb)$ is $P$-equivariant, hence it corresponds to a subgroup $Q'\subseteq P$ containing $J$, by~\cite[Theorem~7.6]{Pe17}.

Similarly as above, the Schubert prime divisors of $P/Q'$ correspond to the set of simple roots
\[
S\smallsetminus\left(S^p(P/J)\cup\{\beta\}\right)
\]
and the complementary $S^p(P/J)\cup\{\beta\}$ is the set of simple roots of the Levi subgroup $M'$ of $Q'$ containing $T$.

If $D$ is a color, the set of simple roots of $M'$ does not contain the set of simple roots of $M$, thus contradicting the uniqueness of $Q$. The claim is proven in this case.

We suppose now that $D$ is not a color. In this situation the irreducible spherical module $U$ has a highest weight vector that is $G$-stable, therefore $U$ has dimension $1$, it intersects $D$ in the point $0\in U$, and the restriction $\pi|_D\colon D\to P/Q$ is a $G$-equivariant bijection between $G$-homogeneous varieties, so it is an isomorphism. In other words $P/J$ is the total space of a line bundle over $P/Q$ and $D$ is the image of the zero section $s=\left(\pi|_D\right)\inv \colon P/Q\to D\subseteq P/J$.

Let $p\in P$ be an element sending the origin $0\in U$ to a non-zero point of $U$, then $p\cdot s$ is a non-zero section of the line bundle, which implies that the latter is generated by global sections. Therefore it is represented by a divisor on the partial flag variety $P/Q$ that is a linear combination of Schubert divisors with non-negative coefficients. We can write the pull-back along $\pi$ as a divisor of the form
\[
\delta = \sum_{\beta\in S\smallsetminus S^p(P/J)} c_\beta D_\beta
\]
where $c_\beta\geq 0$ for all $\beta$. The image $D$ of the zero section is linearly equivalent to $\delta$ in the total space $P/J$ of the line bundle: it is not difficult to show that any rational function $f\in\CC(P/J)$ expressing this linear equivalence (say, having a simple zero on $D$ and poles of order $c_\beta$ on each $D_\beta$) must be a $B$-eigenvector. The $B$-eigenvalue of $f$ must be $\sigma$, because $\sigma$ takes a positive value on $\rho(D)$ and generates $\Xi(P/J)$. Hence
\[
\<\sigma,\rho(D)\> = 1
\]
and
\[
\<\sigma,\rho(D_\beta)\> =-c_\beta \leq 0
\]
for all $\beta$. This finishes the proof of the claim. 

To complete the proof of~(\ref{lemma:rank1:alpha}), it remains to check the values on $\sigma$ of the simple coroots $\beta^\vee$ where $\beta$ is in $S^p(P/J)$. But these simple coroots vanish on $\Xi(P/J)$, see again~\cite[Section~1.4]{Lu01}.

Finally~\cite[Proposition~6.3]{Pe17} implies $\Colemb_Z=\varnothing$, and at this point Theorem~1.1 in the paper~\cite{GH15} by Gagliardi and Hofscheier implies $S^p(P/J)=S^p(P/H)$. This concludes the proof of~(\ref{lemma:rank1:M}).
\end{proof}

\begin{theorem}\label{thm:NSprecise}
Let $R$ be a connected subgroup of $P$ containing $G$, and let $\sigma\in\Sigma_P(P/H)\smallsetminus\Sigma_R(P/H)$. Then $\sigma\notin\Sigma_G(P/H)$, and the element $w_{0,\sigma}(-\sigma)$ is the highest weight of an irreducible $G$-submodule $\fv$ of $\fp^u$ not contained in $\fr^u$ and such that $V=\Exp(\fv)$ generates a subgroup which does not act trivially on $P/H$.
\end{theorem}
\begin{proof}
We follow the strategy of~\cite[Theorem~6.15]{Pe17}. As in Lemma~\ref{lemma:rank1}, we fix a complete $P$-toroidal $P$-embedding $X$ of $P/H$ and we choose a $G$-orbit $Z$ such that $\Cone_Z$ is contained in $V_P(X)$, has codimension $1$ in $N(X)$ and is orthogonal to $\sigma$.

Consider the valuations cones $V_P\left(\overline Z\right)\subseteq V_R\left(\overline Z\right)\subseteq V_G\left(\overline Z\right)$. The inclusions stem from Corollary~\ref{cor:Pvalcone}. Since $N\left(\overline Z\right)$ has dimension $1$, there are only two possibilities for $V_R\left(\overline Z\right)$, namely $V_P\left(\overline Z\right)$ and $N\left(\overline Z\right)$. In the first case $\Cone_Z$ is contained not only in a face of codimension $1$ of $V_P(X)$ but also of $V_R(X)$, which yields $\sigma\in\Sigma_R(P/H)$: contradiction.

Hence we have $V_R\left(\overline Z\right)=N\left(\overline Z\right)$, in other words $\overline Z$ is $R$-horospherical. The inclusion of $V_R(P/H)$ in $V_G(P/H)$ shows at this point $V_G\left(\overline Z\right)=N\left(\overline Z\right)$, so $\sigma\notin\Sigma_G(P/H)$.

This enables us to apply Lemma~\ref{lemma:rank1}, so let $z$, $J$, and $Q$ be as in that statement. We may suppose again that $G_z$, hence also $J$, contains $B_-$. Since $\overline Z$ is $R$-horospherical, by~\cite[Proposition~6.14]{Pe17} the stabilizer $R_z$ contains $B_-R^u$.

Recall the Levi subgroup $M\subseteq Q$ and the fiber $U=\pi^{-1}(eQ)$ from the proof of Lemma~\ref{lemma:rank1}. We have seen that the lattice $\Xi(Z)$ is generated by $\sigma$, and that it coincides with the lattice $\Xi(U)$ which is generated by the highest weight $\gamma$ of the dual module $U^*$. Hence $\gamma\in\{\sigma,-\sigma\}$.

The unique $(B\cap M)$-stable prime divisor $E$ of $U$ has a valuation that takes a positive value on $\gamma$, and $E$ is obtained intersecting $U$ with the prime divisor $D$ of Lemma~\ref{lemma:rank1} (see the proof of the lemma). Part~(\ref{lemma:rank1:Z}) of the lemma assures at this point $\sigma=\gamma$. So $-\sigma$ is the lowest weight of $U$ and $w_{0,\sigma}(-\sigma)$ is the highest weight.

Finally, consider the intersection $J\cap P^u$: it contains $R^u$, and the key observation is that $J\cap P^u$ is stable under conjugation by $J$ which contains $B_-$. Then $\fj\cap \fp^u$ is a $B_-$-stable proper vector subspace, so it cannot contain all highest weight vectors of $\fp^u$ as a $G$-module. Take any such highest weight vector $v\in\fp^u$ not contained in $\fj\cap \fp^u$, then modulo $\fj\cap \fp^u$ the vector $v$ is a highest weight vector of $\fp^u/(\fj\cap \fp^u)$ as an $M$-module. Therefore $v$ has weight $w_{0,\sigma}(-\sigma)$.

Let $\fv\subseteq \fp^u$ be the irreducible $G$-submodule generated by the highest weight vector $v$. Since $v$ is not contained in $\fj\cap \fp^u$, the image $V=\Exp(\fv)$ generates a subgroup of $P$ not stabilizing the point $z\in \overline Z$. So this group does not act trivially on $\overline Z$, which implies that it is not contained in $R^u$ and correspondingly $\fv$ is not contained in $\fr^u$.
\end{proof}

We apply the above theorem to give informations about the valuation cone of a spherical $G$-variety, knowing a generic stabilizer.

\begin{corollary}
Let $G/K$ be a spherical homogeneous space. Suppose $K$ is contained in a maximal proper parabolic subgroup $Q$ containing $B_-$ corresponding to a simple root $\alpha\in S$. Suppose the Lie algebra $\fq^u$ is an irreducible module under the adjoint action of a Levi subgroup of $Q$. Then $K$ does not contain $Q^u$ if and only if there exists $\sigma\in\Sigma_G(G/K)$ that is not in $\Span_\ZZ S\smallsetminus \{\alpha\}$ and satisfies
\[
w_{0,\sigma}(\sigma)=\alpha.
\]
\end{corollary}
\begin{proof}
Let $M$ be the Levi subgroup of $Q$ containing $T$. First observe that $\alpha$ is not in $S^p(P/H)$ by~\cite[Proposition~3.3.3, part~2)]{Lu01}, so the set of simple roots of $M$ contains $S^p(P/H)$.

The subgroup $K$ contains $Q^u$ precisely when $G/K$ is obtained by a procedure called {\em parabolic induction} from the fiber $Q/K$ of the natural map $G/K\to G/Q$. This is equivalent to the fact that all elements of $\Sigma_G(G/K)$ are in the root lattice of $M$, see~\cite[Section~3.4]{Lu01}. This lattice is $\Span_\ZZ S\smallsetminus \{\alpha\}$.

So if $K$ contains $Q^u$ then all elements of $\Sigma_G(G/K)$ are contained in $\Span_\ZZ S\smallsetminus \{\alpha\}$, and no element $\sigma$ as in the statement of the corollary exists.

Viceversa, if $K$ does not contain $Q^u$ then there is an element $\sigma\in \Sigma_G(G/K)$ that is not in the root lattice of $M$. By Proposition~\ref{prop:fiberoverflag} in this case we have $\sigma\in \Sigma_Q(Q/K)$, and not being in the root lattice of $M$ implies that $\sigma$ is not in $\Sigma_M(Q/K)$. Theorem~\ref{thm:NSprecise} implies that $w_{0,\sigma}(-\sigma)$ is a highest weight appearing in $\fq^u$ as an $M$-module. Under our hypotheses on $Q$, there is only one possibility namely $w_{0,\sigma}(-\sigma)=-\alpha$.
\end{proof}

\begin{example}
The above corollary can be applied to hermitian symmetric varieties. For example, for $\SO(2n)/\GL(n)$ the simple root $\alpha$ is $\alpha_n$ (the short simple root, with the usual Bourbaki numbering), which corresponds to the spherical root $\sigma=\alpha_{n-2}+\alpha_{n+1}+\alpha_n$ if $n$ is odd, and $\sigma=\alpha_n$ if $n$ is even (see~\cite[Section~3, Case~16]{BP15}.
\end{example}

It is a natural question whether one can generalize the above theorem to elements of $\Sigma_G(P/H)$ under suitable assumptions, or prove a form of converse to the theorem. These seem much more delicate issues and here we do not pursue these lines further.

To illustrate the difficulties that arise, we mention that if $P^u$ acts non-trivially on $P/H$ (equivalently $P^u\not\subseteq H$) then it is not true that $\Sigma_P(P/H)$ and $\Sigma_G(P/H)$ must be different, see the next example.

\begin{example}\label{ex:rotto}
Set $G=\SL(n+1)$ and let it act by conjugation on the vector group $P^u=\CC^{n+1}$ with the standard linear action. Consider the resulting semi-direct product $P=G\ltimes P^u$ and its subgroup
\[
H =\left\{ \left(
\begin{pmatrix}
1 & 0 \\
v & *
\end{pmatrix}
,
\begin{pmatrix}
* \\
v
\end{pmatrix}\right)\in G\ltimes P^u \;\middle|\; v\in \CC^{n} \right\}.
\]
Then the unipotent radical $P^u$ is not contained in the subgroup $H$, however $P/H$ is a single $G$-orbit because we have
\[
PH=G\underbrace{P^uH}_{\subseteq GH}= GH.
\]
The variety $P/H= G/(G\cap H)=\SL(n+1)/\SL(n)$ is spherical, and it is possible to prove the equality
\[
\Sigma_G(P/H)=\Sigma_P(P/H)=\{\alpha_1+\ldots+\alpha_n\}
\]
using the same technique of the proof of~\cite[Theorem~13.1]{Pe17}, where $\alpha_1,\ldots,\alpha_n$ are the simple roots of $G$.

The same $G$-variety $\SL(n+1)/\SL(n)$ can be equipped obviously with a different transitive action of $P$ by letting $P^u$ act trivially instead. The two resulting $P$-varieties are not $P$-equivariantly isomorphic, nevertheless they have the same spherical roots for the $P$-action.
\end{example}

\begin{remark}
In the standard theory for reductive groups there is a classification of all possible spherical roots for the action of a given group $G$ on any spherical $G$-variety. We expect that this can be achieved in our setting too if one fixes the group $P$, using Theorem~\ref{thm:NSprecise}, since the elements of $\Sigma_P(P/H)\smallsetminus \Sigma_G(P/H)$ are so closely related to $P^u$.

On the other hand, if only the group $G$ is given and $P$ is allowed to vary, then there are infinitely many possibilities for elements of $\Sigma_P(P/H)\smallsetminus \Sigma_G(P/H)$, see the following examples. The classification of such elements, if possible, would probably involve in a crucial way a description of the branching of representations from $G$ to its reductive subgroups $K$ such that $G/K$ is spherical.
\end{remark}

\begin{example}\label{ex:classificationsr}
Let $G=\SL(2)\times \SL(2)\times \CC^*$ and consider the subgroup $L=\diag(\SL(2))\times\CC^*$. Choose the Borel subgroup $B=B_1\times B_2\times \CC^*$ of $G$, where $B_1=B_2$ is a Borel subgroup of $\SL(2)$, denote by $\delta\colon B\to \CC^*$ the projection onto the third factor, and denote by $\omega_1$ and $\omega_2$ the fundamental dominant weights of the two factors $\SL(2)$ of $G$, extended naturally to characters of $B$. Denote by $\alpha_i=2\omega_i$ the corresponding simple roots.

Let $\CC^2$ be the $L$-module where $\SL(2)$ acts in the usual way and $s\in\CC^*$ acts as the homothety of factor $s$. The equivariant vector bundle
\[
X=G\times^L \CC^2
\]
is a spherical $G$-variety, with spherical roots $\Sigma_G(X)=\{\alpha_1,\alpha_2\}$. This stems from the fact that a generic stabilizer in $G$ of $X$ corresponds to case~8 of Table~A in~\cite{Wa96}.

Suppose now $V=V(\lambda)$ is an irreducible $G$-module with highest weight $\lambda$ such that $\CC^2$ is an $L$-submodule of $V$. Using the sphericity of $X$ one can show in general that the multiplicity of $\CC^2$ in $V$ must be $1$; however in this case it is an easy consequence of the Clebsch-Gordan formula. Then we have a unique $L$-equivariant projection $\pi\colon V\to \CC^2$.

Set $P=G\ltimes P^u$ where $P^u$ is the vector group that, equipped with the action of $G$ by conjugation, is $G$-equivariantly isomorphic to $V$. To avoid confusion we fix a $G$-equivariant isomorphism $\ell\colon P^u\to V$, we will use the notation $g*u=gug\inv$ for the conjugation action of $g\in G$ on $u\in P^u$, and the juxtaposition $gv$ if we mean the action of $g\in G$ on $v\in V$.

The action of $G$ on $X$ can be extended to $P$ by setting
\[
gu\cdot [x,v] = \left[g\cdot x,v+\pi\left(\ell\left(x\inv* u\right)\right)\right]
\]
where $u\in P^u$, $g,x\in G$, and $v\in \CC^2$. In this way $P^u$ acts non-trivially and $P$ acts transitively on $X$.

Notice that $X$ is not a single $G$-orbit: the image $Z$ of the zero section of the vector bundle is a $G$-stable but not $P$-stable closed subvariety. This implies that not all $G$-invariant valuations of $\CC(X)$ are $P$-invariant, hence $\Sigma_G(X)\neq \Sigma_P(X)$. Since the strict inclusion $\Sigma_G(X)\supsetneq \Sigma_P(X)$ is impossible because the inclusion $V_G(X)\supseteq V_P(X)$ holds, there exists $\sigma\in\Sigma_P(X)\smallsetminus \Sigma_G(X)$.

Theorem~\ref{thm:NSprecise} yields $\lambda=w_{0,\sigma}(-\sigma)$. The possibilities for $\lambda$ are $\lambda=\mu+\delta$ with
\[
\mu=n\omega_1+(n-1)\omega_2
\]
or
\[
\mu=(n-1)\omega_1+n\omega_2
\]
where $n$ is any positive integer.

It remains to compute $w_{0,\sigma}$. We have $S^p(P/H)=\varnothing$ from Table~A in~\cite{Wa96}, but $w_{0,\sigma}$ also depends on which simple coroot is strictly positive on $\sigma$. From the easy structure of the Weyl group of $G$ it is elementary to conclude
\[
\sigma = \pm\mu-\delta
\]
regardless of the value of $n$.

We will see in Example~\ref{ex:noncosimplicial} that representation theory yields $\sigma=-\mu-\delta$. The other possibility $\mu-\delta$ is not excluded easily using embeddings, because it is contained in the convex cone generated by $\Sigma_P(X)$.

We can apply these arguments to any $\lambda$ such that $V$ contains $\CC^2$, obtaining various spherical roots $\sigma$ for various different groups $P$:
\[
\sigma=-n\omega_1-(n-1)\omega_2-\delta
\]
and
\[
\sigma=-(n-1)\omega_1-n\omega_2-\delta
\]
can be spherical roots for any positive integer $n$.

It is interesting to notice the unusual shape of these spherical roots, in contrast with the spherical roots of the $G$-action which are always dominant weights for this particular $G$.
\end{example}

\begin{example}\label{ex:toric}
Let $G=(\CC^*)^n$ be a torus. Set $P^u=\CC$, and let $G$ act by conjugation on $P^u$ as scalar multiplication via a non-trivial character $\lambda\colon G\to \CC^*$. Let $X$ be a toric $G$-variety, in which case $V_G(X)=N(X)$ and $\Sigma_G(X)=\varnothing$ because $G$ has no non-trivial unipotent subgroups. Suppose the action of $G$ extends to $P=G\ltimes P^u$, in which case $\lambda$ is called a {\em Demazure root} of $X$. Then it is a standard fact that $X$ is not a single $G$-orbit, so the difference $\Sigma_P(X)\smallsetminus \Sigma_G(X)$ is not empty and Theorem~\ref{thm:NSprecise} implies
\[
\Sigma_P(X)=\{-\lambda\}.
\]
This was obtained already in~\cite[Example~9.1]{Pe17} using a different method.
\end{example}

\section{Morphisms, automorphisms, subvarieties}\label{s:morphisms}

In~\cite{Kn91} equivariant regular maps between spherical $G$-embeddings are discussed.

\begin{definition}[\cite{Kn91}]
Let $G/M$, $G/N$ be spherical $G$-homogeneous varieties with $M\subseteq N$ and let $\varphi\colon G/M\to G/N$ be the natural map. Pulling back to $G/M$ the $B$-eigenvectors in $\CC(G/N)$ induces a homomorphism $\varphi_*\colon N(G/M)\to N(G/N)$. Denote by $\Colemb_\varphi$ the set of $G$-colors of $G/M$ mapped dominantly to $G/N$.
\begin{enumerate}
\item A colored cone $(\Cone,\Colemb)$ in $N(G/M)$ {\em maps} to a colored cone $(\Cone',\Colemb')$ of $N(G/N)$ if we have
\begin{enumerate}
\item $\varphi_*(\Cone)\subseteq \Cone'$, and
\item $\varphi(\Colemb\smallsetminus\Colemb_\varphi)\subseteq \Colemb'$.
\end{enumerate}
\item A family $\Fan$ of colored coned in $N(G/M)$ {\em maps} to a family $\Fan'$ of colored cones in $N(G/N)$ if each element of $\Fan$ maps to some element of $\Fan'$.
\end{enumerate}
\end{definition}

Theorem~4.1 in~\cite[Theorem~4.1]{Kn91} states that the map $\varphi$ of the above definition extends to a $G$-equivariant map $X\to X'$ between respective $G$-embeddings if and only if $\Fan(X)$ maps to $\Fan(X')$. In this case, if $Z\subseteq X$ is a $G$-orbit then the colored cone $(\Cone_Z,\Colemb_Z)$ maps to $\left(\Cone_{\varphi(Z)},\Colemb_{\varphi(Z)}\right)$ (this follows actually from the same theorem, applied to simple $G$-embeddings).

Let us work on the same problem, but with $P$-embeddings. Let $K\subseteq P$ be a subgroup containing $H$, we want to determine when the induced natural map $\varphi\colon P/H\to P/K$ extends to respective $P$-embeddings $X$ and $X'$.

We apply Knop's theorem to the respective open $G$-orbits of $X$ and $X'$, and we obtain that $\varphi$ extends to a $G$-equivariant map $X\to X'$ if and only if $\Fan(X)$ maps to $\Fan(X')$. Such extension is automatically $P$-equivariant, because $\varphi$ itself is $P$-equivariant, so the situation is already well understood.

Our contribution is that actually it is enough to test $\Fan(X)^\uparrow$ and $\Fan(X')^\uparrow$ instead of $\Fan(X)$ and $\Fan(X')$.

\begin{theorem}
Let $K\subseteq P$ be a subgroup containing $H$ and let $\varphi\colon P/H\to P/K$ the corresponding natural map. Let $X$ and $X'$ be $P$-embeddings respectively of $P/H$ and $P/K$. Then $\varphi$ extends to a $P$-equivariant regular map $X\to X'$ if and only if $\Fan(X)^\uparrow$ maps to $\Fan(X')^\uparrow$.
\end{theorem}
\begin{proof}
Suppose $\varphi$ extends to $\varphi\colon X\to X'$. Then $\Fan(X)$ maps to $\Fan(X')$ as we have recalled. It remains to show that each $(\Cone,\Colemb)\in\Fan(X)^\uparrow$ maps to some element of $\Fan(X')^\uparrow$.

This holds because each such $(\Cone,\Colemb)$ is of the form $(\Cone,\Colemb)=(\Cone_Z,\Colemb_Z)$ for a $G$-orbit $Z$ with $P$-stable closure, so $\varphi(Z)\subseteq X'$ is a $G$-orbit with $P$-stable closure, therefore $\left(\Cone_{\varphi(Z)},\Colemb_{\varphi(Z)}\right)$ is in $\Fan(X')^\uparrow$.

Viceversa, suppose that $\Fan(X)^\uparrow$ maps to $\Fan(X')^\uparrow$. The proof here is similar to the proof of~\cite[Theorem~4.1]{Kn91}.

We may assume that $X$ and $X'$ are $P$-simple, so $\Fan(X)^\uparrow$ has a unique maximal element $(\Cone,\Colemb)$, and it maps to the unique maximal element $(\Cone',\Colemb')$ of $\Fan(X')^\uparrow$. 

Consider $X_0\subseteq X$ as in the proof of Theorem~\ref{thm:existencePsimple} and $X'_0\subseteq X'$ defined similarly. Since $(\Cone,\Colemb)$ maps to $(\Cone',\Colemb')$, it is easy to check that $\varphi^*$ sends $\CC[X_0]^{(B)}$ into $\CC[X'_0]^{(B)}$. Together with~\cite[Corollary~1.7]{Kn91} this implies that $\varphi^*$ sends $\CC[X_0]$ into $\CC[X_0]$, so $\varphi$ induces a regular map $X_0\to X'_0$ and by equivariance a map $X=PX_0\to X'=PX'_0$. 
\end{proof}

We conclude this section with two additional results, both generalizing known facts from the reductive case: the linear part of the valuation cone is related to the $P$-equivariant automorphisms of $P/H$ (see e.g.~\cite[Theorem~6.1]{Kn91}), and the valuation cone of $P$-stable subvarieties is computed easily as a projection of $V_P(P/H)$ (see~\cite[Theorem~1.1]{GH15}).

\begin{proposition}
The group of $P$-equivariant automorphisms of $P/H$ is a diagonalizable group, of dimension equal to the dimension of the linear part $V_P(P/H)\cap (-V_P(P/H))$ of $V_P(P/H)$.
\end{proposition}
\begin{proof}
Subgroups $K\subseteq P$ containing $H$ and such that $K/H$ is connected correspond to $P$-equivariant colored subspaces in $N(P/H)$, thanks to~\cite[Theorem~7.6]{Pe17}. At this point the proof of Theorem~6.1 in~\cite{Kn91} holds verbatim with our $P$ instead of the reductive group $G$ of loc.cit., yielding the proposition.
\end{proof}

\begin{proposition}
Let $X$ be a $P$-embedding of $P/H$, and let $Z\subseteq X$ be a $P$-stable subvariety. Then the lattice $\Xi(Z)$ is the common vanishing locus in $\Xi(X)$ of all elements of $\Cone_Z$, and we have
\[
\pi_*\left(V_P(P/H)\right) = V_P(Z)
\]
where $\pi_*\colon N(X)\to N(Z)$ is the natural map dual to the inclusion $\Xi(Z)\to \Xi(X)$.
\end{proposition}
\begin{proof}
The statement regarding $\Xi(Z)$ is well-known for spherical varieties. It remains to prove the equality $\pi_*\left(V_P(P/H)\right) = V_P(Z)$.

Suppose first that $X$ is also complete and $P$-toroidal. In~\cite[Proposition~6.6]{Pe17} it is shown that $V_P(Z)$ is the image in $N(Z)$ of the union
\[
\bigcup_W\Cone_W \subseteq N(X)
\]
where $W$ varies in the set of $G$-orbits with $P$-stable closure and satisfying $\Cone_W\supseteq \Cone_Z$. For any $G$-orbit $V$ with $P$-stable closure, the segment joining any element of $\Cone_V$ and any element of $\Cone_Z$ must be in $V_P(X)$, since $V_P(X)$ is convex. We deduce that the relative interior of this segment must intersect some cone of the fan of the form $\Cone_W$ where $W$ is as above. At this point it is elementary to deduce that $V_P(Z)$ is simply the image $\pi_*(V_P(X))$.

Let now $X$ be an arbitrary $P$-embedding of $P/H$. As already done above, let $X'$ be a $P$-equivariant completion of $X$ and let $X''$ be a complete $P$-toroidal embedding of $P/H$ such that the identity of $P/H$ extends to a proper $P$-equivariant regular map $\varphi\colon X''\to X'$. Let $Z''\subseteq X''$ be a $G$-orbit such that $\varphi\left(Z''\right)=Z$; any such $Z''$ clearly has $P$-stable closure, and we already noticed that $(\Cone_{Z''},\Colemb_{Z''})$ maps to $(\Cone_Z,\Colemb_Z)$, in particular $\Cone_Z''$ is contained in $\Cone_Z$.

Since $X''$ is complete and $P$-toroidal, the union of all cones appearing in $\Fan\left(X''\right)^\uparrow$ coincides with $V_P(P/H)$. On the other hand the colored cone $(\Cone_Z,\Colemb_Z)$ is $P$-invariant, so the relative interior of $\Cone_Z$ intersects $V_P(P/H)$. As a consequence, we may choose $Z''$ as above, with the additional requirement that $\Cone_{Z''}$ has the same dimension of $\Cone_Z$.

The first part of the proof yields the equalities $\Xi(Z)=\Xi\left(Z''\right)$ and $V_P\left(Z''\right)=\pi_*\left(V_P(P/H)\right)$, it remains to show $V_P\left(Z''\right)=V_P(Z)$. This stems from~\cite[Proposition~7.8]{Pe17}, which is stated under the assumption that the restriction $\varphi\colon Z''\to Z$ has connected fibers, but the same proof of loc.cit. holds verbatim even without this assumption.
\end{proof}

\section{The ring of regular functions as a $P$-module}\label{s:tails}

Let $X$ be a spherical $G$-variety. If $X$ is affine, one defines the {\em weight monoid} $\Gamma(X)$ of $X$ to be the set of highest weights appearing in $\CC[X]$ as a $G$-module. It is a submonoid of the monoid of all dominant weights, and corresponds to the decomposition
\[
\CC[X] = \bigoplus_{\lambda\in\Gamma(X)} V(\lambda)
\]
of $\CC[X]$ in irreducible summands, where $V(\lambda)$ has highest weight $\lambda$. With respect to this decomposition, in general $\CC[X]$ may fail to be a graded ring, and it is well-known that the set of spherical roots $\Sigma_G(X)$ provides a ``global measure'' of this failure.

In the next proposition we give a version of this result where $X$ is a $P$-embedding. In this case the set of spherical roots $\Sigma_P(X)$ is involved.

\begin{theorem}\label{thm:Ptails}
Suppose $X$ is an affine $P$-embedding of $P/H$. Let $\T\subseteq \Xi(X)\otimes_\ZZ\QQ$ be the convex cone generated by all $B$-weights that can be written as
\[
\lambda + \mu -\nu
\]
where $V(\nu)$ is contained in the product%
\footnote{By this notation we mean the the submodule generated by all products $pv_1\cdot qv_2$ where $v_1\in V(\lambda)$, $v_2\in V(\mu)$, and $p,q\in P$.}
$(P\cdot V(\lambda))\cdot(P\cdot V(\mu))$ inside $\CC[X]$. Then $\T$ is the convex cone generated by $\Sigma_P(X)$.
\end{theorem}
\begin{proof}
Let $f\in V(\nu), f_1\in V(\lambda), f_2\in V(\mu)$ be highest-weight vectors, with $\lambda,\mu,\nu$ as in the statement. Such $f$ is a linear combination of the form
\[
f = \sum_{i=1}^n (a_i p_i \cdot f_1)\cdot (b_iq_i\cdot f_2) 
\]
where $a_i,b_i\in\CC$ and $p_i,q_i\in P$. Then any $P$-invariant valuation satisfies
\[
\nu(f)\geq \nu(f_1)+\nu(f_2)
\]
so
\[
\nu\left(\frac{f_1 f_2}{f}\right)\leq 0.
\]
This shows that $\T$ is contained in the cone generated by $\Sigma_P(X)$.

For the opposite inclusion we go to the dual cones, and we prove the inclusion $\T^\vee\subseteq V_P(X)$. Let $v\in\T^\vee$ be a non-zero vector, define $\Cone = \QQ_{\geq 0}v$ and $\Colemb=\varnothing$. The proof of Theorem~\ref{thm:existencePsimple} can be carried out for $(\Cone,\Colemb)$ almost verbatim. Precisely, the function $f_0$ of that proof can be choosen in $\CC[P/H]$ since $P/H$ is quasi-affine, so $V$ of loc.cit.\ is actually a $P$-submodule of $\CC[P/H]$, and the inequality $\<\chi, c\>\geq 0$ of loc.cit.\ follows from the definition of $\T$.

The construction given by that proof produces a $P$-embedding $Y$ of $P/H$ with a $P$-stable prime divisor whose valuation is a positive rational multiple of $v$: this shows that $v$ is in $V_P(P/H)=V_P(X)$.
\end{proof}

\begin{example}
Let $X=\CC$ where $P=\CC^*\ltimes \CC$, $H=G=\CC^*$ acting with weight $1$ on $\CC$. Then $P/H$ is spherical under the action of $G$, with open $G$-orbit $\CC\smallsetminus\{0\}$ and $P^u=\CC$ acting by translations. Then
\[
\CC[X] = \bigoplus_{n\geq 0} \CC \cdot z^n
\]
where $\CC\cdot z^n = V(-n)$ as usual. The smallest $P$-submodules are
\[
\CC \cdot 1 \subsetneq \CC \cdot 1 \oplus \CC z \subsetneq \CC[X].
\]
Notice that $V(-1)^*$ is isomorphic to $P^u$ as a $G$-module, on the other hand $V(-1)$ is not a $P$-submodule of $\CC[X]$. Example~\ref{ex:toric} shows that $-1$ is the $P$-spherical root of $X$, and we notice that this corresponds to the fact that the $P$-submodule generated by $V(-1)$ is $V(0)\oplus V(-1)$, and we see the spherical root $-1$ exactly as the difference between the two weights.
\end{example}

If we set $P=G$ in Theorem~\ref{thm:Ptails} then we recover one of the standard versions of the result. We underline that it is a ``global'' information: given specific weights $\lambda,\mu$, $\nu$, knowing that $\lambda+\mu-\nu$ is in the cone $\T$ does not assure that $V(\nu)$ appears as a summand of $V(\lambda)\cdot V(\mu)$ in $\CC[X]$. A general combinatorial criterion characterizing whether it appears seems one of the most challenging open problems in the theory of spherical varieties.

Interestingly, if $P$ contains $G$ strictly instead, the action of $P$ does provide some more local information which is not available in the purely reductive case. But this information is related to $G$-submodules in $\CC[X]$ that are not $P$-submodules, instead of being related to the multiplication in $\CC[X]$.

\begin{proposition}\label{prop:Ptails}
Let $X$ be an affine $P$-embedding of $P/H$, suppose $H^u\subseteq P^u$, let $Z\subseteq X$ be the unique closed $G$-orbit, and let $\lambda\in\Gamma(X)$. If $\lambda\notin\Gamma(Z)$ then the $G$-submodule $V(\lambda)\subseteq \CC[X]$ is not $P$-stable, i.e., there exists $\nu\in\Gamma(X)$ with $\nu\neq \lambda$ such that $V(\nu) \subseteq P\cdot V(\lambda)$.
\end{proposition}
\begin{proof}
Let $L$ be a Levi subgroup of $H$ contained in $G$, then $X$ is $G$-equivariantly isomorphic to the equivariant vector bundle
\[
X \cong G\times^L (P^u/H^u)
\]
and the closed $G$-orbit $Z$ is $G$-equivariantly isomorphic to $G/L$.

It is a standard fact that the irreducible submodule $V(\lambda)\subseteq \CC[X]$ is $P$-stable if and only if the unipotent radical $P^u$ acts trivially on $V(\lambda)$. This is also equivalent to $V(\lambda)$ being contained in $\CC[X]^{P^u}=(\CC[P]^H)^{P^u}=\CC[P]^{HP^u}=\CC[P/(HP^u)]=\CC[G/L]$. So in our case $V(\lambda)$ is not $P$-stable.
\end{proof}

\begin{example}\label{ex:noncosimplicial}
We construct an example where the set of spherical roots $\Sigma_P(X)$ is not linearly independent over $\QQ$, so the convex cone $V_P(X)$ is not cosimplicial.

Let $G$ and $X$ be as in Example~\ref{ex:classificationsr}, and consider the $G$-module
\[
V=V(\omega_1+\delta)\oplus V(\omega_2+\delta).
\]
Thanks to our particular choice of weights, each summand is itself $L$-equivariantly isomorphic to $\CC^2$. For $i\in\{1,2\}$ let $\pi_i\colon V\to \CC^2$ be the projection along the other summand.

Define as before $P=G\ltimes P^u$ but with $P^u\cong V$, and let $P$ act on $X$ by setting
\[
gu\cdot [x,v] = \left[gx,v+\pi_1\left(\ell\left(x\inv *u\right)\right)+\pi_2\left(\ell\left(x\inv *u\right)\right)\right].
\]
As in Example~\ref{ex:classificationsr} one shows that the difference $\Sigma_P(X)\smallsetminus \Sigma_G(X)$ is not empty and contained in the set $\{\pm\omega_1-\delta,\pm\omega_2-\delta\}$. The inclusion $V_G(X)\supseteq V_P(X)$ implies that $\Sigma_G(X)=\{\alpha_1,\alpha_2\}$ is contained in the convex cone generated by $\Sigma_P(X)$.

Fix a highest weight vector $\eta$ of the dual of $\CC^2$, considered as an $L$-module and choosing the Borel subgroup $\diag(B_1)\times \CC^*$ of $L$. We define two regular functions $f_1,f_2\in \CC[X]$ as
\[
\begin{matrix}
f_i\colon & X & \to & \CC\\
& [(x_1,x_2,s),v] & \mapsto & \eta(x_isv)
\end{matrix}
\]
with $x=(x_1,x_2,s)\in G$. It is elementary to check that $f_i$ is well-defined and a $B$-eigenvector of $B$-eigenvalue $\omega_i-\delta$, which is therefore an element of the weight monoid of $X$. Let $u_i\in P^u$ be such that $\ell(u_i)$ is a highest weight vector of the summand $V(\omega_{i}+\delta)$ of $V$, and let us apply $u_{3-i}$ to the function $f_i$. We obtain the function
\[
(u_{3-i}\cdot f_i)\colon [(x_1,x_2,s),v] \mapsto  \eta\left(x_1s\left(v-\pi_1\left(\ell\left(x\inv *u_{3-i}\right)\right)-\pi_2\left(\ell\left(x\inv * u_{3-i}\right)\right)\right)\right).
\]
This function is the difference of $f_i$ itself and the function
\[
\begin{matrix}
F_i\colon & X & \to & \CC\\
& [(x_1,x_2,s),v] & \mapsto & \eta\left(x_is\left(\pi_1\left(\ell\left(x\inv *u_{3-i}\right)\right)+\pi_2\left(\ell\left(x\inv *u_{3-i}\right)\right)\right)\right).
\end{matrix}
\]
We have
\[
\pi_{3-i}\left(\ell\left(x\inv *u_{3-i}\right)\right) = \pi_{3-i}\left(x\inv \ell(u_{3-i})\right)=x\inv \ell(u_{3-i})=
(x_{3-i})\inv \ell(u_{3-i})
\]
because $x\inv \ell(u_{3-i})$ is already in $V(\omega_{3-i}+\delta)$ and because the $i$-th factor of $G$ acts trivially on $V(\omega_{3-i}+\delta)$. So $F_i$ is actually the function
\[
[(x_1,x_2,s),v]  \mapsto \eta\left(x_is (x_{3-i})\inv s\inv \ell(u_{3-i})\right).
\]
One checks easily that $F_i$ is a $B$-semiinvariant of weight $\omega_1+\omega_2$ and that it is not the zero function%
\footnote{Compare this to a similar computation yelding $u_i\cdot f_i=f_i$ instead.}.

This yields
\[
P\cdot V(\omega_i-\delta) \supseteq V(\omega_i-\delta)\oplus V(\omega_1+\omega_2)
\]
in the ring $\CC[X]$. Thanks to Theorem~\ref{thm:Ptails} applied to $\lambda=\omega_i-\delta$ and $\mu=0$, for all $i$ the difference
\[
\omega_{i}-\delta- (\omega_1+\omega_2)=-\omega_{3-i}-\delta
\]
is in the convex cone generated by $\Sigma_P(X)$.

Finally, the element $-\omega_i-\delta$ is not contained in the convex cone generated by $\alpha_1$, $\alpha_2$, $\omega_i-\delta$, $\pm\omega_{3-i}-\delta$, so $-\omega_i-\delta$ must be an element of $\Sigma_P(X)$ for all $i$. Analogously, the simple root $\alpha_i$ is not in the convex cone generated by $\alpha_{3-i}$, $\pm\omega_1-\delta$, $\pm\omega_{2}-\delta$, so $\alpha_i$ must be an element of $\Sigma_P(X)$.

Since the elements $\omega_i-\delta$ are in the convex cone generated by $\alpha_1$, $\alpha_2$, $-\omega_1-\delta$, $-\omega_2-\delta$ instead, we conclude
\[
\Sigma_P(X)=\{\alpha_1,\alpha_2,-\omega_1-\delta,-\omega_2-\delta\}.
\]

A similar procedure yields
\[
\Sigma_P(X)=\{\alpha_1,\alpha_2,-\mu-\delta\}
\]
for the same variety $X$ and the group $P$ of Example~\ref{ex:classificationsr}.
\end{example}

\section{Characterization of log-homogeneous embeddings}\label{s:Plog}

\begin{definition}[\cite{Br07} and~\cite{BDP90}]
Let $P$ be a connected linear algebraic group, and let $X$ be a smooth irreducible $P$-variety.
\begin{enumerate}
\item Let $D\subseteq X$ be a divisor with normal crossings. Denote by $\T_X$ the tangent sheaf of $X$ and by
\[
\T_X(-\log D)
\]
the subsheaf given by the derivations that preserve the ideal sheaf $\O_X(-D)$. Suppose the action of $P$ preserves the divisor $D$, and consider the natural map
\[
\op_{X,D}\colon \O_X\otimes \fp \to \T_X(-\log D)
\]
induced by the action, where $\fp$ is the Lie algebra of $P$. Then the pair $(X,D)$ is {\em $P$-homogeneous} if $\op_{X,D}$ is surjective. In this case $D$ is uniquely determined by $X$ and $P$, being the union of all $P$-stable prime divisors (see Remark~2.1.1 in~\cite{Br07}), hence we will also simply say that $X$ is {\em log $P$-homogeneous}.
\item We say that $X$ is {\em $P$-regular} if the following conditions holds:
\begin{enumerate}
\item the group $P$ has on $X$ an open orbit, whose complement is a union of irreducible divisors with normal crossings, called {\em boundary divisors},
\item any $P$-orbit closure in $X$ is the transversal intersection of some of the boundary divisors,
\item for all $x\in X$ the normal space $T_x X/T_x(P\cdot x)$ contains an open orbit of the stabilizer $P_x$.
\end{enumerate}
\end{enumerate}
\end{definition}

We recall some results by Brion from \cite{Br07} on these two notions.

\begin{theorem}\label{thm:brion}
Let $P$ be a connected linear algebraic group with Levi subgroup $G$ and let $X$ be a smooth irreducible $P$-variety.
\begin{enumerate}
\item (Corollary~2.1.4 in \cite{Br07}.) If $X$ is $P$-regular then it is log $P$-homogeneous.
\item\label{thm:brion:loghom} (Part of Theorem~3.2.1 in \cite{Br07}.) If $X$ is complete and log $P$-homogeneous, then
\begin{enumerate}
\item\label{thm:brion:loghom:spherical} the variety $X$ is $G$-spherical, and
\item\label{thm:brion:loghom:XL} there exists an open $G$-stable and $G$-regular subset $X_G\subseteq X$ such that every $P$-orbit of $X$ intersects $X_G$ along a unique $G$-orbit.
\end{enumerate}
\end{enumerate}
\end{theorem}

We complete the above picture in the following theorem.

\begin{theorem}\label{thm:Plog}
Let $P$ be a connected linear algebraic group with Levi subgroup $G$, and let $X$ be a smooth complete irreducible $P$-variety. Then $X$ is log $P$-homogeneous if and only if $X$ is $G$-spherical and $P$-toroidal.
\end{theorem}
\begin{proof}
Suppose $X$ is log $P$-homogeneous. Then $X$ is $G$-spherical by Theorem~\ref{thm:brion}, part~(\ref{thm:brion:loghom:spherical}). Let $X_G$ be as in Theorem~\ref{thm:brion}, part~(\ref{thm:brion:loghom:XL}). We claim that $X_G$ is $G$-toroidal, and this follows from the proof of Lemma~3.1.2 in~\cite{Br07} and Proposition~2.2.1 in the paper~\cite{BiBr96} by Bien and Brion.

More precisely, the local structure of $X_G$ is explained in details in~\cite{Br07}: for a point $x$ of any closed $G$-orbit, one takes the stabilizer $Q\subseteq G$ of $x$. It is a parabolic subgroup of $G$, and we can assume it contains $B$. A neighborhood of $x$ in $X_G$ is $Q$-equivariantly isomorphic to a product $Q^u\times \AA^n$, where a Levi subgroup $M$ of $Q$ acts by conjugation on $Q^u$ and linearly on $\AA^n$ via $n$ linearly independent characters. The unipotent radical $Q^u$ acts by translation only on the first factor.

As in~\cite{BiBr96}, one notices then that the coordinate hyperplanes of this affine space $\AA^n$, i.e.\ the $M$-stable prime divisors, are the intersections of $\AA^n$ with the $G$-stable prime divisors of $X_G$ containing $x$, and there are no other $B$-stable prime divisors in this open chart. This shows that any $B$-stable prime divisor containing $x$ is $G$-stable, therefore $X_G$ is $G$-toroidal. Finally, the embedding $X$ is $P$-toroidal thanks to~\cite[Corollary~6.4]{Pe17}.

Viceversa, suppose $X$ is $G$-spherical and $P$-toroidal. Then there exists an open $G$-stable subset $X_G$ such that $X_G$ is a $G$-toroidal embedding of the open $G$-orbit of $X$, and $X_G$ intersects any $P$-orbit of $X$ in a single $G$-orbit, thanks to~\cite[Corollary~6.4]{Pe17}. The local structure of $G$-toroidal varieties, such as $X_G$ here, is well-known to be as above again. It follows easily, as in~\cite{BiBr96}, that $X_G$ is $G$-regular.

By the second part of Theorem~3.2.1 in~\cite{Br07} we conclude that $X$ is log $P$-homogeneous.
\end{proof}

\end{document}